\documentclass[reqno, 11pt, a4paper]{amsart}
\allowdisplaybreaks
\usepackage{a4wide}
\usepackage{comment}
\usepackage{amssymb}
\usepackage{centernot} 
\usepackage{bbm}
\usepackage{libertine} 
\usepackage[cal=euler, scr=boondoxo]{mathalfa}
\usepackage[colorlinks=true, linkcolor=blue,citecolor=blue]{hyperref}
\usepackage{booktabs}
\usepackage{epigraph}
\usepackage{graphicx}
\usepackage{mathtools}
\mathtoolsset{showonlyrefs,showmanualtags}
\newtheorem{theorem}{Theorem}[section]
\newtheorem{lemma}[theorem]{Lemma}
\newtheorem{proposition}[theorem]{Proposition}

\theoremstyle{definition}

\theoremstyle{remark}
\newtheorem{remark}[theorem]{Remark}

\numberwithin{equation}{section}
\def\N{{\mathbb N}}
\def\Z{{\mathbb Z}}
\def\Q{{\mathbb Q}}
\def\R{{\mathbb R}}

\newcommand{\dd}{\text{\normalfont d}}

\newcommand{\abs}[1]{\ensuremath{\left\lvert #1 \right\rvert}} 
\newcommand{\norm}[1]{\ensuremath{\left\lVert #1 \right\rVert}} 

\newcommand{\E}{{\mathbb E}}
\renewcommand{\P}{{\mathbb P}}

\newcommand{\IND}{\mathbbm{1}}

\newcommand{\cB}{\ensuremath{\mathcal B}}
\newcommand{\cC}{\ensuremath{\mathcal C}}
\newcommand{\cD}{\ensuremath{\mathcal D}}

\newcommand{\cJ}{\ensuremath{\mathcal J}}
\newcommand{\cK}{\ensuremath{\mathcal K}}
\newcommand{\cL}{\ensuremath{\mathcal L}}
\newcommand{\cM}{\ensuremath{\mathcal M}}

\newcommand{\cP}{\ensuremath{\mathcal P}}

\newcommand{\cS}{\ensuremath{\mathcal S}}
\newcommand{\cT}{\ensuremath{\mathcal T}}

\newcommand{\cW}{\ensuremath{\mathcal W}}
\newcommand{\cX}{\ensuremath{\mathcal X}}
\newcommand{\cY}{\ensuremath{\mathcal Y}}
\newcommand{\cZ}{\ensuremath{\mathcal Z}}
\newcommand{\bbN}{\ensuremath{\mathbb N}}

\newcommand{\bbQ}{\ensuremath{\mathbb Q}}

\newcommand{\bfE}{\ensuremath{\mathbf E}}

\newcommand{\bfP}{\ensuremath{\mathbf P}}

\begin{document}
	\title[Fractional kinetics from a Markovian system of interacting BTM]{Fractional kinetics equation from a Markovian system of interacting Bouchaud trap models}
	
	\author{Alberto Chiarini}
	\address{Università degli Studi di Padova}
	\curraddr{}
	\email{chiarini@math.unipd.it}
	\thanks{}
	
	
	\author{Simone Floreani}
	\address{Universit\"at Bonn}
	\email{sflorean@uni-bonn.de}
	\curraddr{}
	\thanks{}
	
	
	\author{Federico Sau}
	\address{Università degli Studi di Trieste}
	\curraddr{}
	\email{federico.sau@units.it}
	\thanks{}
		
		\begin{abstract}
			We consider a partial exclusion process evolving on $\Z^d$ in a random trapping environment. In	 dimension $d\ge 2$, we derive the fractional kinetics equation  
			\begin{equation*}\frac{\partial^\beta\rho_t}{\partial t^\beta}  =  \Delta \rho_t
			\end{equation*}
			as a hydrodynamic limit of the particle system. Here,  $\frac{\partial^\beta}{\partial t^\beta}$, $\beta\in(0,1)$, denotes the fractional  derivative in the Caputo sense. We  thus exhibit a Markovian interacting particle system whose frequency field rescales to a sub-diffusive equation corresponding to a non-Markovian process, the Fractional Kinetics process.	
			In contrast, we show that, when $d=1$, the  system rescales to the solution to
			\begin{equation*}
				\frac{\partial \rho_t}{\partial t}= \cL_\beta \rho_t\ ,
			\end{equation*}
			where  $\cL_\beta$ is the random generator of the singular quasi-diffusion known as the FIN diffusion.
		\end{abstract}
	\date{}
	\dedicatory{}
	\maketitle
	\thispagestyle{empty}



\section{Introduction}
Deriving macroscopic equations from  microscopic stochastic dynamics is a well established topic in the literature of interacting particle systems (IPS), and this limiting procedure is known as taking a hydrodynamic limit (HL); see, e.g., \cite{kipnis_scaling_1999}. One of the most studied models is the simple symmetric exclusion process (SSEP), a system of simple random walks on $\Z^d$  jumping to nearest-neighboring sites at rate one and subject to the exclusion rule: at most one particle per site is allowed and 	jumps to  occupied sites are suppressed. The HL for this system is well-known: when diffusively rescaled, the empirical density field of  SSEP converges, in a proper sense, to a deterministic measure, absolutely continuous with respect to the Lebesgue measure and whose density is the solution of the heat equation. The derivation of the HL  for SSEP is easier compared to that for other particle systems  (e.g., \cite[Ch.\ 5]{kipnis_scaling_1999}) because SSEP satisfies stochastic self-duality \cite{liggett_interacting_2005-1}. This means that
the expected evolution of the product of $k$ occupation variables can be studied by looking at SSEP with $k$ particles only. In particular, the expectation of the empirical density field reduces to  an expectation with respect to one simple symmetric random walk, and proving the HL of SSEP amounts to understanding the scaling limit under a parabolic rescaling of this simple symmetric random walk. This reduction is still valid when looking at exclusion processes in specific random environments, and in this article we are going to present a new interacting particle system evolving in a random trapping landscape for which the HL can be obtained from the scaling limit of the associated single particle model.
\subsection{IPS in random environment}
In recent years, there has been an upsurge of activity  around particle systems in random environments (e.g., \cite{nagy_symmetric_2002,goncalves_scaling_2008,jara_quenched_2008,redig_symmetric_2020,floreani_HDL_2021}). A random environment is an extra source of randomness added to the model, aiming at capturing the effects of impurities or inhomogeneities of the underlying medium in which the particles evolve. One of the most studied random environments is obtained by attaching symmetric random weights $\omega_{xy}$, named random conductances, to the edges  of the graph on which the particles hop.
Specializing our discussion to SSEP,   HL with random conductances have been derived in several works under various conditions. For instance, assuming that the conductances  $\omega_{xy}$  are sampled according to a probability which is  stationary and ergodic under translation in $\Z^d$ and that the expectations of both $\omega_{xy}$ and $\omega_{xy}^{-1}$  are finite, A.\ Faggionato \cite{faggionato_bulk_2007,faggionato_cluster_2008,faggionato2022hydrodynamic} and M.\ Jara \cite{jara_hydrodynamic_2011} proved independently that the quenched HL of SSEP is the heat equation with a non-degenerate diffusivity that depends only on the law of the environment. Moreover, the results in \cite{faggionato2022hydrodynamic} go beyond the context of $\Z^d$,  allowing for more general underlying graphs, while those in \cite{faggionato_hydrodynamic_2009,FrancoLandim2010, Valentim2011} weaken the conditions on the random conductances so to obtain sub-diffusive equations known as quasi-diffusions in the limit.
In contrast, several physical systems display time-fractional sub-diffusive behavior at the macroscopic scale (e.g., \cite{Sokolov_etAL02}, and references therein), and  little is known about microscopic \emph{interacting} particle systems related to macroscopic time-fractional	sub-diffusivity. The aim of this paper is to introduce a microscopic interacting dynamics that properly rescaled exhibits such time-fractional macroscopic behavior.
\subsection{Model}
The particle system that we consider consists of interacting Bouchaud trap models (see, e.g., \cite{ben-arous_cerny_dynamics_2006}), generalizing the inhomogeneous partial  exclusion process that was introduced in \cite{floreani_HDL_2021}. We start by recalling the definition of the Bouchaud trap model (BTM), which  is  a random walk in random environment introduced in \cite{Bou92} as a model displaying trapping phenomena, as well as aging. The random environment is given by a collection  of i.i.d.\ $\N$-valued  random variables $\alpha=(\alpha_x)_{x\in\Z^d}$ (here and all throughout, $\N:=\{1,2,\ldots\}$), which, letting $\Q$ denote the corresponding product measure, satisfy 
\begin{equation}\label{eq:heavy-tail-0}
\Q\big(\alpha_0> u\big)= u^{-\beta}\left(1+o(1)\right)\ ,\quad \text{as}\ u\to \infty\ ,	
\end{equation}   
for some $\beta \in (0,1)$. We remark that having $\N$-valued  random variables is not the standard choice in the context of trap models, but it will be necessary for the particle system that we will consider.

For every $a\in [0,1]$, the \textit{Bouchaud trap model with parameter $a\in [0,1]$} (${\rm BTM}(a)$) in the environment $\alpha$ sampled according to $\Q$ is the continuous-time random walk $(X_t)_{t\ge 0}$ on $\Z^d$ with infinitesimal generator given, for all $g: \Z^d\to \R$, by
\begin{equation}\label{eq:generator-BTM}
A g(x):=  \alpha_x^{a-1}
\sum_{y\sim x}   \alpha_y^a\left(g(y)-g(x)\right)\ ,\qquad x \in \Z^d\ ,
\end{equation}
where the above summation runs over $y \in \Z^d$ such that $|y-x|=1$. One refers to the case $a=0$ as the \textit{symmetric} case while to	 $a\in(0,1]$ as the \textit{asymmetric} cases. Note that the symmetric version of ${\rm BTM}$  consists of a symmetric simple random walk on $\Z^d$ with  exponential holding times, whose means are i.i.d.\ and heavy-tailed distributed.

In this work, we analyze a system of infinitely-many ${\rm BTM}(a)$, which, instead of evolving as independent particles, are subject to a partial-exclusion rule. In more detail, given  $\Q$ as in \eqref{eq:heavy-tail-0} and $\alpha=(\alpha_x)_{x\in\Z^d}$ sampled according to $\Q$, we consider the particle system $(\eta_t)_{t\ge 0}$ on the state space $\Xi:=\prod_{x\in\Z^d} \{0,1,\ldots, \alpha_x\}$, and whose  infinitesimal evolution is described by the  (formal) Markov generator $\cL$,  acting on local functions (i.e., functions that depend on configurations only through their values over a finite number of sites)  $\varphi:\Xi\to \R$ as follows: 
\begin{align}\label{eq:generator-IPS}
\cL \varphi(\eta)= \sum_{x\in \Z^d}\eta(x)\, \alpha_x^{a-1}\sum_{y\sim x} \alpha_y^a\left(1-\eta(y)/\alpha_y\right) \left(\varphi(\eta^{x,y})-\varphi(\eta)\right)
\ ,\qquad \eta \in \Xi\ .
\end{align}
Here, for a configuration $\eta \in \Xi$ and sites $x, y \in \Z^d$, $\eta(x)$ represents the number of particles in $x\in \Z^d$, while $\eta^{x,y}:= \eta-\delta_x+\delta_y\in \Xi$	 denotes the particle configuration obtained from $\eta\in \Xi$ by removing a particle from $x\in \Z^d$ (if any) and placing it on $y\in \Z^d$ (if not already full, i.e., $\eta(y)<\alpha_y$). 
Note that the term  $-\eta(y)/\alpha_y$ in \eqref{eq:generator-IPS} introduces an exclusion-like interaction among particles, which otherwise would evolve as independent copies of ${\rm BTM}(a)$ as described in \eqref{eq:generator-BTM}. Thus, the variable $\alpha_x$ has two interpretations in this model. On the one side, it enters in the jump rates of the particles; on the other side, it represents the maximal occupancy at site $x$ (or the size of the trap at site $x$; see also Figure \ref{fig:imagetraps}).
\begin{figure}
	\centering
	\includegraphics[width=0.85\linewidth]{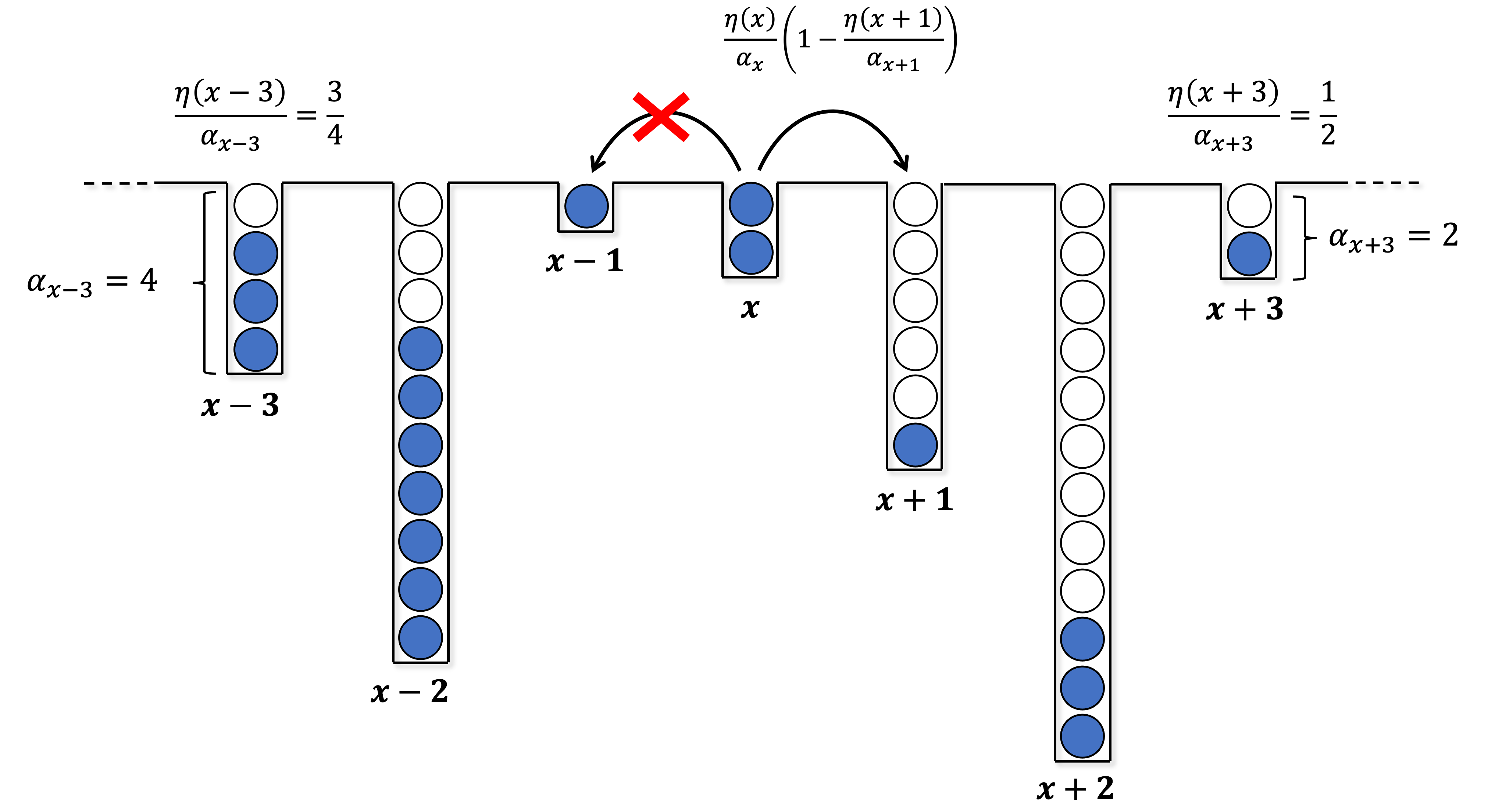}
	\caption{Schematic representation of the interacting system of ${\rm BTM}(a)$ on $\Z$ for a given realization of the environment, and with $a=0$. In each site $y\in\Z$,  there is a trap of size $\alpha_y$; when the trap is full, no particles are allowed to jump there.}
	\label{fig:imagetraps}
\end{figure}
\subsection{Summary of main results}
Denote by $\P_\eta$ the quenched law of the interacting particle system above when starting from $\eta$. The main result of the paper concerns the asymptotic behavior as $n\to \infty$ of $(\cZ^n_t)_{t\ge 0}$, the empirical field for a space-time rescaling of $(\eta_t)_{t\ge 0}$,  given, for some  $\theta_n\in (0,\infty)$ (this scaling factor $\theta_n$ is the relevant time scale for a single ${\rm BTM}(a)$; see \eqref{eq:theta-n} and the subsequent text), by
\begin{equation}\label{eq: frequency field intro}\cZ^n_t:=	\frac{1}{n^d}\sum_{x\in \Z^d} \frac{\eta_{t\theta_n}(x)}{\alpha_x}\, \delta_{x/n}\ , \qquad n \in \bbN\ .\end{equation}
In what follows, we refer to such field as empirical frequency field (see \S\ref{sec:HL} for more details).
Interestingly, the scaling and the  limiting behavior change with the dimension, as we now briefly summarize (we refer to Theorem \ref{th:HDL} below for the precise statement).

Start by considering $d\ge 3$, and set $\theta_n=\theta_{n,\beta,d\ge3}:= n^{2/\beta}$. Fix a  macroscopic profile, that is, a uniformly continuous function $\bar \rho: \R^d\to [0,1]$; furthermore, for every collection $\alpha=(\alpha_x)_{x\in \Z^d}$ and $n\in \N$, let the process $\eta_t$ be initialized at $t=0$ according to the local equilibrium measure $\mu_n  = \mu_{\alpha,\bar \rho}^n= \otimes_{x\in \Z^d}\,{\rm Bin}(\alpha_x,\bar \rho(x/n))$.  Then,  for $\Q$-a.e.\ $\alpha$, the field $\cZ^n_t$
 converges to a deterministic measure absolutely continuous with respect to the Lebesgue measure and with deterministic density $(\rho_t)_{t\ge 0}$ solving, on $\R^d$,
\begin{equation}\label{eq: FK intro}
\frac{\partial^\beta}{\partial t^\beta} \rho_t = D_{\mathrm{eff}}\Delta \rho_t\ ,\quad \text{with}\ \rho_0=\bar \rho\ ,
\end{equation}
where $D_{\mathrm{eff}}=D_{\mathrm{eff}}(\Q,a,d)\in (0,\infty)$ does not depend on the specific realization of the environment $\alpha$.
Equation \eqref{eq: FK intro} --- with the fractional derivative in time of order $\beta\in(0,1)$ meant in the Caputo sense (see \eqref{eq: caputo derivative} below) --- is a well-studied time-fractional sub-diffusive PDE, known as \textit{fractional kinetics equation} (FKE), e.g., \cite{chen_time-2017}. 
In words, our main result states that the HL of the partial exclusion process in a heavy-tailed random environment, when sub-diffusively rescaled, is a deterministic measure, absolutely continuous with respect to the Lebesgue measure and whose density is deterministic and solves the FKE. As we show in Theorem \ref{th:HDL} below, an analogous result holds in dimension $d=2$, but time has to be rescaled differently. More precisely, the same claim holds in $d=2$ for $\cZ^n_t$ given in \eqref{eq: frequency field intro}, but with $\theta_n$ now given by 
$ \theta_n=\theta_{n,\beta,d=2}:= n^{2/\beta}(\log n)^{1-1/\beta}$.

We emphasize that, for the field $\cZ_t^n$ and for $d\ge 2$,  the randomness of the environment as well as the randomness of the microscopic dynamics homogenize when passing to the macroscopic scale. 
This is not the case if one looks at the scaling behavior of a different field, namely the most typically studied empirical density field, 	 defined as 
\begin{equation}\label{eq: density field intro}\cX^n_t=\frac{1}{n^{d/\beta}}\sum_{x\in \Z^d}\eta_{t\theta_n}(x)\,\delta_{x/n}\ .
\end{equation}
Indeed, when initializing the process as done above, under $\Q(\dd \alpha)\mu_n(\dd \eta)\P_\eta$ and for all $t\ge0$,   $\cX^n_t$
converges in distribution (with respect to a suitable topology) as $n\to \infty$ to $\rho_0\cW$, where  $\rho_0\cW(\dd x)= 	\bar \rho (x)\cW(\dd x)$, and
$\cW$ is the random measure on $\R^d$ given by
\begin{equation}\label{eq:W}
	\cW=\cW_\beta:= \sum_{i}v_i\,	\delta_{x_i}\ ,
	\end{equation}
	with $\{(v_i,x_i)\}_i$ being the support of a Poisson point process on $(0,\infty)\times \R^d$ with intensity measure $\beta v^{-(1+\beta)}\dd v\dd x$. In other words, the randomness of the environment still survives in the  reference measure, but the macroscopic dynamics is frozen. In a few words, this mechanism is explained by the fact that the fields $\cX_t^n$ only capture  those huge, randomly located, piles of particles, which, however, do not manage to significantly move over times of the order of $\theta_n$.

We also analyze the one-dimensional case and find a much different behavior which involves both a different  rescaling and  limiting equation. In this case, under $\Q(\dd \alpha)\mu_n(\dd \eta)\P_\eta$, the field $\cX^n_t$ given in \eqref{eq: density field intro}, but with $\theta_n$ now defined as $\theta_n=\theta_{n,\beta,d=1}:= n^{1+1/\beta}$,	
converges in distribution  as $n\to \infty$ to $\rho_t\cW$, where $\cW$ is the same as before, while $\rho_t$ is random and solves, on $\R$, 
\begin{equation}\label{eq:d=1}\frac{\partial}{\partial t} \rho_t(x)=\frac{\partial}{\partial \cW}\frac{\partial}{\partial  x}\rho_t(x)\ ,\quad \text{with}\ \rho_0=\bar \rho\ .
\end{equation}
Here, $\frac{df}{d\cW}=g$ has to be interpreted as
\begin{equation}\label{eq:df/dw}
f(x)= c +\int_{-\infty}^x g(z)\, \cW (\dd z)\ ,\quad x \in \R\ ,
\end{equation}
for some $c\in \R$.
In words, in $d=1$, our main result states that the HL of the partial exclusion process in a heavy-tailed random environment, when sub-diffusively rescaled, is a random measure, absolutely continuous with respect to $\cW$ in \eqref{eq:W} and whose density is random  and solves the random quasi-diffusion PDE given in \eqref{eq:d=1}.
 We obtain an analogous result  for $\cZ_t^n$ in $d=1$, with the one-dimensional Lebesgue measure replacing the random measure $\cW$ above (see~\eqref{eq:theorem-d=1} in Theorem~\ref{th:HDL}).

\subsection{Comparison with the literature}
The only interacting system which has been shown to rescale to a sub-diffusive equation is SSEP with i.i.d.\ conductances satisfying, for some  $\beta \in (0,1)$,
\begin{equation*}
\Q(\omega_{xy}^{-1}>u)=u^{-\beta}(1+o(1))\ ,\quad\text{as}\ u\to \infty\ .
\end{equation*}
When properly rescaling the empirical density field, it has been proven for $d=1$ in \cite{faggionato_hydrodynamic_2009}  (see \cite{FrancoLandim2010,Valentim2011} for the more general case with $d\ge 1$) that the HL is given by 
\begin{equation*}\frac{\partial}{\partial t} \rho_t(x)=\frac{\partial}{\partial x}\frac{\partial}{\partial W}\rho_t(x)\ ,\quad x \in \R\ ,
\end{equation*}
where $W$ is a random \textit{càdlàg} function with heavy-tailed jumps of parameter $\beta$, and  $\frac{df}{d W}=g$ is interpreted as in \eqref{eq:df/dw} with $W$ in place of $\cW$; for further details, we refer the interested reader to \cite{faggionato_hydrodynamic_2009}.

From the point of view of one-particle models, the literature on	 random walks converging to sub-diffusive processes is vast, and a prominent place is taken by the ${\rm BTM}$ (see \cite{ben-arous_cerny_dynamics_2006}, and references therein).  Scaling limits of ${\rm BTM}$ have been studied (see, e.g., \cite{FIN02,BenArousCerny05, BarlowCerny, cernyEJP}), and the limiting processes are the so-called Fontes-Isopi-Newman (FIN) diffusion in $d=1$, and the fractional kinetics process in $d\ge 2$. 

The article~\cite{JaraLandimTeixeira} is the only one dealing with rigorous scaling limits for systems of many ${\rm BTM}$. There,  the authors consider a system of \emph{independent} symmetric particles of this kind and, in $d=1$, obtain 
\eqref{eq:d=1} as a HL, i.e., the same HL that arises for the  one-dimensional partial exclusion process in a heavy-tailed random environment as considered in this paper.
Interacting ${\rm BTM}$ have not been studied so far: our model is the first example in this direction and it has the advantage of still satisfying stochastic duality.

 Let us emphasize that all previous works on hydrodynamic limits for many trap models went through the well established and developed formulation via martingale problems, combined with one- and two-block estimates as in, e.g., \cite{kipnis_scaling_1999}. While this formulation proved to be extremely effective in the derivation of local (in time) PDEs in the last three decades, these techniques do not easily extend  to time-fractional limits (or, equivalently, to evolution equations associated to non-Markovian processes). In contrast, our approach   via duality and stochastic representations of the solutions to time-fractional PDEs provides a direct connection between the pre-limit and the hydrodynamic limit, flexible enough so to include time-fractional setups.

\subsection{Organization of the paper}
The rest of the paper is organized as follows. In \S\ref{section: model and results}, we rigorously state the main result, while
\S\ref{sec:proof-th:main-general} is devoted to its proof. In   Appendix \ref{section: construction}, we gather all the details on the rigorous construction of the infinite particle system.

\section{Main result}\label{section: model and results}
In what follows, $\N=\{1,2,\ldots\}$. Moreover, $\cC_b(\R^d)$ denotes the space of continuous and bounded functions on $\R^d$,  $\cC_c^+(\R^d)$ its subspace of non-negative functions with compact support, while $\cM_v(\R^d)=(\cM_{loc}(\R^d),\tau_v)$ denotes the space of locally finite measures on $\R^d$ endowed with the vague topology $\tau_v$, i.e., the weakest topology on $\cM_{loc}(\R^d)$ for which all mappings $\nu \in \cM_{loc}(\R^d)\mapsto \langle \nu|f\rangle:=\int_{\R^d}f(x)\, \nu(\dd x) \in \R$,  $f\in \cC_c^+(\R^d)$, are continuous. Throughout, we write \textquotedblleft$Y^n\underset{n\to \infty}\Longrightarrow Y$\textquotedblright\  for convergence in distribution of real-valued random variables $Y^n,Y\in \R$, and \textquotedblleft$\cY^n\underset{n\to \infty}\Longrightarrow \cY$ in $\cM_v(\R^d)$\textquotedblright\ for convergence in distribution of random measures $\cY^n,\cY\in \cM_v(\R^d)$, i.e., 
\begin{align}\label{eq:conv-measures}
\cY^n\underset{n\to \infty}\Longrightarrow \cY\quad {\rm in}\ \cM_v(\R^d)\qquad \text{if}\qquad  \langle \cY^n|f\rangle \underset{n\to \infty}\Longrightarrow \langle \cY|f\rangle\ \text{for all}\ f\in \cC^+_c(\R^d)\ .
\end{align} 

Let us first introduce the underlying random environment. 
Fix $\beta \in (0,1)$ and $d\ge 1$. 
Recall from \eqref{eq:W} the definition of the random measure \begin{equation}
	\cW=\cW_\beta:= \sum_i v_i\, \delta_{x_i} \in \cM_v(\R^d)\ ,
\end{equation}  generated by the Poisson point process $\{(v_i,x_i)\}_i$ on $(0,\infty)\times \R^d$ with intensity measure $\beta v^{-(\beta+1)}\dd v \dd x$.  We define, for all $n \in \N$,  
\begin{equation}\label{eq:Wn}
\cW^n=\cW^n_\beta:= \frac{1}{n^{d/\beta}} \sum_{x\in \Z^d} \alpha_x\, \delta_{x/n} \in \cM_v(\R^d)\ ,
\end{equation}
where $(\alpha_x)_{x\in \Z^d}$ is a collection of $\N$-valued i.i.d.\ random variables with product law $\Q=\Q_\beta$ whose marginals satisfy
\begin{equation}\label{eq:heavy-tail}
\Q\big(\alpha_0> u\big)= u^{-\beta}\left(1+o(1)\right)\ ,\quad \text{as}\ u\to \infty\ .	
\end{equation} 
By inspection of the moment generating functions, the following convergence in distribution in $\cM_v(\R^d)$ holds (see, e.g., \cite[Lem.\ 5.3]{croydon2017time}): recalling the definition of   $\cW$ from \eqref{eq:W}, we have
\begin{align}\label{eq:conv-W}
\cW^n\underset{n\to \infty}\Longrightarrow \cW\ \quad \text{in}\ \cM_v(\R^d)\ .
\end{align}

Before presenting the interacting particle system, we recall the definition of a single ${\rm BTM}$, and some of its main properties.

\subsection{${\rm BTM}$ and its scaling limits}\label{sec:BTM}
Fix $a\in [0,1]$.
Given $\Q$ satisfying \eqref{eq:heavy-tail},  the \textit{Bouchaud trap model with parameter $a\in [0,1]$} (${\rm BTM}(a)$) in the environment $\alpha$ sampled according to $\Q$ is the continuous-time random walk $(X_t)_{t\ge 0}$ on $\Z^d$ with  infinitesimal generator $A$ given in \eqref{eq:generator-BTM}.
It is well-known (see, e.g., \cite{BarlowCerny} and references therein) that, $\Q$-a.s., the ${\rm BTM}(a)$	 does not explode in finite time a.s.\ when starting from any locations $x\in \Z^d$.
 Note that $A$  is
symmetric in $L^2(\cW^n)$ with $n=1$.  In what follows, $\mathbf P^\alpha_x$ and $\mathbf E^\alpha_x$ denote the quenched law and corresponding expectation of $X_t$ when $X_0=x\in \Z^d$.

We are mainly interested in the scaling properties of ${\rm BTM}(a)$. For this purpose, let, for all $n \in \N$ and continuous and bounded functions $g\in \cC_b(\R^d)$,
\begin{equation}\label{eq:Pnt}
P^n_t g(x/n):= 	\mathbf E^\alpha_x\left[g(X_{t\theta_n}/n)\right]\ ,\qquad x \in \Z^d\ ,\ t \ge 0\ ,	
\end{equation}
denote the  semigroup of a sub-diffusive rescaling of $X_t$, where space has been shrunk by a factor $1/n$, and time 	sped-up by the following amount:
\begin{equation}\label{eq:theta-n}
\theta_n=\theta_{n,\beta,d}:=\begin{cases}
n^{2/\beta} &\text{if}\ d \ge 3\\
n^{2/\beta}\left(\log (1+n)\right)^{1-1/\beta} &\text{if}\ d=2\\
n^{1+1/\beta} &\text{if}\ d=1\ .
\end{cases}
\end{equation}
As we will discuss in \S\ref{sec:proof-th:main-general2a} and \S\ref{sec:proof-th:main-general2c}, there exists some strictly positive constant $\kappa=\kappa(\Q,\beta,a,d)$, which does not depend on the realization of the environment, for which $P^n_t$ converges, in a suitable sense, to $\mathcal P_{\kappa t}$ as $n\to \infty$, where, for every $g \in \cC_b(\R^d)$, $g_t:=\mathcal P_t g$ solves the following  PDEs, on $\R^d$,
\begin{equation}
	\begin{aligned}
		&\frac{\partial}{\partial t}g_t(x) = \frac{\partial}{\partial \cW}\frac{\partial}{\partial x}g_t(x)\ , &&\qquad d=1\ ,\\
		&\frac{\partial^\beta}{\partial t^\beta}g_t(x) = \Delta g_t(x)\ , &&\qquad d \ge 2\ ,
	\end{aligned}
\end{equation}
with $g_0=g$.  We recall that one writes $\frac{\partial}{\partial \cW}\frac{\partial}{\partial x}f=h$  if 
\begin{equation}\label{eq:generator-FIN}
f(x)=c_1 +c_2x +\int_0^x\int_0^y h(z)\cW(\dd z)\, \dd y\ ,\qquad x \in \R\ ,
\end{equation}
holds for some $c_1,c_2\in \R$, while $\frac{\partial^\beta}{\partial t^\beta}$ is meant in the Caputo sense (e.g., \cite[Eq.\ (1.1)]{chen_time-2017}):	
\begin{equation}\label{eq: caputo derivative}
\frac{\partial^\beta}{\partial t^\beta} f(t):=\frac{1}{\Gamma(1-\beta)}\int_0^t \frac{1}{(t-s)^\beta}f'(s)\, 	\dd s\ ,\qquad t\ge 0\ ,\ f \in \cC^1(\R)\ .
\end{equation}

Interestingly, $(\mathcal P_t)_{t\ge 0}$ admits a stochastic representation:
\begin{itemize} \item  for $d=1$, it is the (random) Markov semigroup of  the ${\rm FIN}$ diffusion $(Z(t))_{t\ge 0}$ of parameter $\beta \in (0,1)$ associated to the random measure $\cW=\cW_\beta$  in \eqref{eq:W};
	\item for $d\ge 2$, it is the pseudo-semigroup of the semi-Markov process $(FK_{\beta}(t))_{t\ge 0}$  known as fractional kinetics process of parameter $\beta \in (0,1)$.
\end{itemize}
More precisely, letting $B=(B(t))_{t\ge 0}$ denote the $d$-dimensional Brownian motion, for $d=1$, the FIN diffusion with parameter $\beta \in (0,1)$ is given by
\begin{equation}
Z(s):=B(\psi^{-1}(s))\ ,\qquad s \ge 0\ ,
\end{equation}
where $\psi{^{-1}}$ denotes the generalized right inverse of 
$t\mapsto\psi(t):=\int_\R \ell(t,x)	\cW(\dd x)$, with  $\ell(t,x)$ being the local time of $B(t)$ at $x\in \R$.
In other words, $Z$ is a diffusion process (actually,  often referred to as a quasi- or singular diffusion, see, e.g., \cite{ben-arous_cerny_dynamics_2006}) expressed as a time change of a standard one-dimensional Brownian motion  with speed measure $\cW$.  As for 	 $FK_\beta$ in $d\ge2$, we have
\begin{equation}\label{eq: FK process}
FK_{\beta}(t):=B(V_\beta^{-1}(t))\ ,\qquad t \ge 0\ ,
\end{equation} 
where  $V_\beta$ stands for a $\beta$-stable subordinator independent of $B$, and \begin{equation*}V_\beta^{-1}(t):=\inf\{s\ge 0:V_\beta(s)>t\}\ .
	\end{equation*}
Finally, we remark that both processes introduced above are sub-diffusive, i.e., their mean square displacement scales as 
\begin{equation*}E[Z(t)^2]=t^{2\beta/(1+\beta)}\quad \text{and}\quad E[FK_\beta(t)^2]=t^\beta\ ,\qquad t\ge 0\ . 
\end{equation*}
For further properties of $Z(t)$ and $FK_\beta(t)$, we refer the interested reader to \cite[\S\S3,4]{ben-arous_cerny_dynamics_2006} and references therein.

\subsection{Interacting particle system} Recall the interacting system of ${\rm BTM}(a)$ with (formal) generator $\cL$ given in \eqref{eq:generator-IPS}. For $\Q$-a.e.\ environment $\alpha$, the existence of the associated particle system with finitely many particle is an immediate consequence of the a.s.\ non-explosiveness of the ${\rm BTM}(a)$ starting from all positions and a stirring representation of the exclusion system (for more details, we refer to Appendix \ref{sec:appendix-second-construction}). The corresponding infinite particle system $\eta_t$ turns out to be  well-defined, $\Q$-a.s.,  for all times and initial configurations as being the unique Markov process on $\Xi$ whose joint moments satisfy the following  identities, a generalization of the most classical duality relations for the usual symmetric exclusion process (see, e.g., \cite[Ch.\  VIII, Th.\ 1.1]{liggett_interacting_2005-1}): for all $k\in \N$, for all  $\eta,\xi\in \Xi$ with $\xi$ being a finite configuration, and for all $t\ge 0$,
\begin{equation}\label{eq:dual-rel-k}
	\E_\eta\big[D(\xi,\eta_t)\big]=  \E_\xi\big[D(\xi_t,\eta)\big]\ ,
\end{equation}
where
\begin{equation}
	D(\xi,\eta):= \prod_{x\in \Z^d} \frac{\eta(x)!}{(\eta(x)-\xi(x))!}\frac{(\alpha_x-\xi(x))!}{\alpha_x!}\IND_{\xi(x)\le \eta(x)}\ .
\end{equation} 
Again, for more details on this construction, we refer to Appendix \ref{sec:appendix-first-construction}.

Fix a realization of $(\alpha_x)_{x\in \Z^d}$ according to $\bbQ$. For a probability measure $\mu$ on $\Xi$,  $\P^\alpha_\mu$ and $\E^\alpha_\mu$ denote the law and corresponding expectation, respectively, of the particle system $\eta_t$ such that  $\eta_0 \sim \mu$.
By a detailed balance computation, it is not difficult to show that, for every $\rho \in [0,1]$, the product measure $\nu_\rho:= \otimes_{x\in \Z^d}\, {\rm Bin}(\alpha_x,\rho)$ is reversible for $\eta_t$.

\subsection{Hydrodynamic limit}\label{sec:HL}
In order to present our main result, we introduce:
\begin{enumerate}
	\item  a uniformly continuous function  $\rho_0: \R^d\to[0,1]$,  playing the role of initial macroscopic datum;
	\item \label{it:2}	the  measures  $ \rho_t{\rm Leb}_{\R^d}(\dd x)=  \rho_t(x)	 {\rm 	Leb}_{\R^d}(\dd x)$ and $ \rho_t\cW(\dd x)=	   \rho_t(x)\cW(\dd x)$ in $\cM_v(\R^d)$, where $ {\rm Leb}_{\R^d}$ denotes the Lebesgue measure on $\R^d$, $\cW$ is random and given in \eqref{eq:W}, and  \begin{equation}\label{eq:rho-t}\rho_t:= \mathcal P_{\kappa t}\rho_0\ ,
		\end{equation} for $\kappa=\kappa(\Q,\beta,a,d)>0$ as given in \S\ref{sec:BTM} above;
	\item two  fields associated to the sub-diffusive rescaling of $\eta_t$ as in	 \eqref{eq:Pnt}:
	\begin{equation}\label{eq:frequency-fields}
	t \in [0,\infty)\longmapsto \cZ^n_t:=	\frac{1}{n^d}\sum_{x\in \Z^d} \frac{\eta_{t\theta_n}(x)}{\alpha_x}\,	\delta_{x/n}\in \cM_v(\R^d)\ .
	\end{equation}
	and 
	\begin{equation}\label{eq:density-fields}
	t \in [0,\infty)\longmapsto	\cX^n_t:= \frac{1}{n^{d/\beta}}\sum_{x\in \Z^d} \eta_{t\theta_n}(x)\, \delta_{x/n} \in \cM_v(\R^d)\ ,
	\end{equation}
	where $\theta_n$ is defined in \eqref{eq:theta-n}.
\end{enumerate}
While $\cX^n_t$ is the well-studied empirical density field, $\cZ^n_t$ is less standard and we refer to it as \textit{empirical frequency field}. There are two main differences between these two fields. On the one hand, $\cZ^n_t$ considers the frequency variables (thus, justifying the name) of occupied slots at each site, rather than the occupancy variables themselves; we remark that frequency as a standard observable appears extensively in the population genetics literature, where, e.g., sites represent colonies, particles individuals of type $A$, say, and holes individuals of type $a$ (see, e.g., \cite{etheridge2011}). On the other hand, for $\cZ^n_t$, each Dirac delta	 is weighted by the most standard weight $n^{-d}$, the unit volume of a cube in $\R^d$ of size $1/n$.

Recall \eqref{eq:conv-measures}. We now present the main result of the paper. 

\begin{theorem}[Hydrodynamic limit]\label{th:HDL} 	Let $a\in [0,1]$. For all $n \in \N$ and $\alpha =(\alpha_x)_{x\in \Z}$, let \begin{equation}\label{eq:product-local-eq}\mu_n=\otimes_{x\in \Z^d}\, {\rm Bin}(\alpha_x,\rho_0(x/n))
		\end{equation} be the initial distribution of the particle system. Recall $\rho_t$ from \eqref{eq:rho-t}. Then, for all  $t\ge0$, we have, for $d\ge 2$,  \begin{equation}\label{eq:theorem-d>2}\cZ^n_t\underset{n\to \infty}\Longrightarrow \rho_t\,{\rm Leb}_{\R^d}\quad \text{and}\quad \cX^n_t\underset{n\to \infty}\Longrightarrow \rho_0\cW\quad \text{in $\cM_v(\R^d)$}\ ,\end{equation}
		while, for $d=1$,  \begin{equation}\label{eq:theorem-d=1}
			\cZ^n_t\underset{n\to \infty}\Longrightarrow \rho_t\,{\rm Leb}_\R\quad \text{and} 	\quad \cX^n_t\underset{n\to \infty}\Longrightarrow \rho_t\cW\quad \text{in $\cM_v(\R)$}\ .
			\end{equation}
\end{theorem}
 The limiting objects in \eqref{eq:theorem-d>2}--\eqref{eq:theorem-d=1} for the empirical density field $\cX_t^n$ are random, while those  for the frequency field $\cZ_t^n$ are deterministic in $d\ge 2$ and random in $d=1$. Moreover,  Theorem \ref{th:HDL} states that such convergences hold in law, namely,  keeping into account the joint randomness of the underlying environment $\alpha \sim \Q$ and  of the particle dynamics. Nevertheless, we will show (Remark \ref{remark: quenched vanish noise}) that, for $d\ge 2$, the convergence for $\cZ_t^n$ is \textit{quenched}, i.e., for $\Q$-a.e.\ $\alpha$,  $\cZ_t^n$ converges in law (thus, in probability) to the deterministic measure $\rho_t{\rm Leb}_{\R^d}$, namely, 
 \begin{equation}\label{eq:quenched-result}
\Q\text{-a.s.}\ ,\qquad \lim_{n\to \infty}\P_{\mu_n}\left(\left| \langle \cZ_t^n|f\rangle- \langle \rho_t\,{\rm Leb}_{\R^d}|f\rangle\right|>\delta\right)=0\ ,
 \end{equation} 
holds true for all $\delta >0$, $t\ge 0$ and $f\in \cC_c^+(\R^d)$. Finally, we discuss the (non-essential) role of the initial measure $\mu_n$ in \eqref{eq:product-local-eq} in Remark \ref{rem:prod-measure} below.

\section{Proof of Theorem \ref{th:HDL}}\label{sec:proof-th:main-general}

In this section we prove  Theorem \ref{th:HDL}, with  $\mu_n=\otimes_{x\in \Z^d}\, {\rm Bin}(\alpha_x,\rho_0(x/n))$. In what follows, for $x\in \Z^d$ and $t\ge 0$, we set
\begin{equation}\label{eq:duality1}
	D^n_t(x/n):=\eta_{t\theta_n}(x)/\alpha_x\ \quad \text{and}\quad \rho^n_t(x/n):= \E_{\mu_n}\left[D^n_t(x/n)\right]\ .
\end{equation}
Moreover, we set $\left\|f\right\|_{n,p}^p:= \langle \cW^n| |f|^p\rangle$ for $p \in [1,\infty)$ (recall $\cW^n$ from~\eqref{eq:Wn}), while we write $\left\|f\right\|_{n,\infty}:=\sup_{x \in \Z^d}\left|f(x/n)\right|$.

  In order to prove the theorem, we decompose the empirical density field $\cX^n_t$, for all $t >0$ and test functions $f \in \cC_c^+(\R^d)$. Recalling the definition of $P_t^n$ from \eqref{eq:Pnt}, we have
\begin{align}\nonumber
\langle \cX^n_t | f\rangle := \frac{1}{n^{d/\beta}}\sum_{x\in \Z^d} \eta_{t\theta_n}(x)\, f(x/n)
\nonumber
&= \frac{1}{n^{d/\beta}} \sum_{x\in \Z^d} \left(\eta_{t\theta_n}(x)-\E_{\eta_0}\left[\eta_{t\theta_n}(x)\right]\right)f(x/n)\\
\nonumber
&\qquad+\frac{1}{n^{d/\beta}}\sum_{x\in \Z^d}\alpha_x \left(\E_{\eta_0}\left[\eta_{t\theta_n}(x)/\alpha_x\right]-\rho^n_t(x/n) \right)f(x/n)\\
\nonumber
&\qquad+\frac{1}{n^{d/\beta}}\sum_{x\in \Z^d}\alpha_x \left( \rho^n_t(x/n)-P^n_t \rho_0(x/n) \right)f(x/n)\\
\nonumber
&\qquad+\frac{1}{n^{d/\beta}}\sum_{x\in \Z^d} \alpha_x\, P^n_t\rho_0(x/n)\, f(x/n)\\
&=: {\rm I}_n+{\rm II}_n+{\rm III}_n+ \langle P^n_t\rho_0\cW^n | f\rangle\ .
\label{eq:decomposition-X}
\end{align}

An analogous decomposition holds for  $\langle\cZ^n_t | f\rangle$. Indeed, recalling the definitions of the fields $\cX^n_t$ and $\cZ^n_t$ in \eqref{eq:density-fields} and \eqref{eq:frequency-fields},  respectively, we readily obtain the following relation between them: for every $\alpha=(\alpha_x)_{x\in \Z^d}$ and $f \in \cC_c^+(\R^d)$, 
\begin{equation}\label{eq:relation-X-Z}
\langle \cZ^n_t|f\rangle = n^{d(1/\beta-1)}\langle \cX^n_t|f^\alpha_n\rangle\ ,
\end{equation}
where
\begin{equation}
f^\alpha_n(x/n):=f(x/n)/\alpha_x\ ,\qquad x \in \Z^d\ .
\end{equation}
Thus,  in view of  \eqref{eq:relation-X-Z} and \eqref{eq:decomposition-X}, we have that 
\begin{align}\label{eq:decomposition-Z}
	\begin{aligned}
\langle \cZ^n_t|f\rangle=n^{d(1/\beta-1)}\big(\tilde{\rm I}_n+\widetilde{\rm II}_n + \widetilde{\rm III}_n\big) + n^{d(1/\beta-1)}\langle P^n_t\rho_0\cW^n| f^\alpha_n\rangle&\\
=: {\rm I}_n^*+ {\rm II}_n^*+{\rm III}_n^* +  n^{d(1/\beta-1)}\langle P^n_t\rho_0\cW^n| f^\alpha_n\rangle&\ , 
\end{aligned}
\end{align} 
where $\tilde{\rm I}_n$ $\widetilde{\rm II}_n$, and $\widetilde{\rm III}_n$ are given as in \eqref{eq:decomposition-X}  with $f^\alpha_n$ in place of $f$.

\smallskip
It now suffices to show the following convergences in distribution (of the environment, jointly with that of the particle system) as $n\to \infty$, for all $t\ge 0$ and $f\in \cC_c^+(\R^d)$:

\smallskip
\noindent
\textbf{Part I.}
 ${\rm I}_n$, ${\rm II}_n$,  ${\rm III}_n$, ${\rm I}_n^*$, ${\rm II}_n^*$, ${\rm III}_n^*$ $\Longrightarrow 0$; 

\smallskip
\noindent
\textbf{Part II(a).} For $d\ge 2$,  $n^{d(1/\beta-1)}\langle P^n_t\rho_0\cW^n| f^\alpha_n\rangle$ $\Longrightarrow$  $\langle \mathcal P_{\kappa t}\rho_0\,{\rm Leb}_{\R^d} |f\rangle$;

\smallskip
\noindent
\textbf{Part II(b).} For $d\ge 2$,  $\langle P^n_t\rho_0\cW^n |f\rangle$ $\Longrightarrow$ $\langle \rho_0\cW |f\rangle$;

\smallskip
\noindent
\textbf{Part II(c).} For $d=1$,  \begin{equation}
	n^{d(1/\beta-1)}\langle P^n_t\rho_0\cW^n| f^\alpha_n\rangle \Longrightarrow  \langle \mathcal P_{\kappa t}\rho_0\,{\rm Leb}_\R |f\rangle\quad\text{and}\quad  \langle P^n_t\rho_0\cW^n |f\rangle\Longrightarrow\langle \mathcal P_{\kappa t}\rho_0\cW|f\rangle\ .\end{equation}
We prove each of these four claims in the following four subsections \S\S\ref{sec:proof-th:main-general1}--\ref{sec:proof-th:main-general2c}.

In the remainder of the paper, we will crucially make use of the stochastic self-duality 
satisfied by the interacting particle system under consideration (see, e.g., \cite{schutz_non-abelian_1994,giardina_duality_2009,redig_sau_2018, floreani_HDL_2021, FloreaniJansenRedigWagner}), which we used to construct the process (Appendix \ref{sec:appendix-second-construction}). In particular, letting 
\begin{equation*}
D(x)= D(x,\eta):=\frac{\eta(x)}{\alpha_x}\quad\text{and}\quad D(x,y)=D((x,y),\eta):=\frac{\eta(x)}{\alpha_x}\frac{\eta(y)-\IND_x(y)}{\alpha_y-\IND_x(y)}\ ;
\end{equation*}
the duality relations in \eqref{eq:dual-rel-k} read as follows: $\Q$-a.s.,  for all $t\ge0$, $\eta \in \Xi$, and  $x, y\in\Z^d$, 
\begin{equation}\label{eq:duality-rel1}\E_{\eta}\left[\eta_t(x)\right]=\alpha_x\, P_tD(x)\ ,
\end{equation}
and
\begin{equation}\label{eq:duality-rel2}\E_{\eta}\left[\eta_t(x)(\eta_t(y)-\IND_x(y))\right]=\alpha_x({\alpha_y-\IND_x(y)})\, P^{(2)}_tD(x,y)\ ,
\end{equation}
where $P_t:=P_t^n$ for $n=1$ denotes the semigroup of ${\rm BTM}(a)$ with generator given in \eqref{eq:generator-BTM}, while $P^{(2)}_t$ the semigroup of the position of two  interacting ${\rm BTM}(a)$ evolving according to the generator given in \eqref{eq:generator-IPS}; the corresponding space-time rescaled versions $P^n_t$ and $P^{(2),n}_t$ are obtained as in \eqref{eq:Pnt}.

\subsection{Proof of Theorem \ref{th:HDL}. Part I}\label{sec:proof-th:main-general1}
\subsubsection{A variance estimate}
We start by proving an auxiliary result. By knowing first and second moments of the empirical density field $\cX^n_t$, and by the negative dependence of the partial symmetric exclusion process, we readily obtain the following variance estimate:
\begin{proposition}\label{pr:variance}
	$\Q$-a.s., for all $\eta \in \Xi$, $t \ge 0$, $n \in \N$, and $f \in \cC^+_c(\R^d)$, 
\begin{equation}\label{eq:negative-dependence}
	\E_\eta\bigg[\bigg( \sum_{x\in \Z^d} \big(\eta_{t\theta_n}(x)-\E_{\eta_0}\left[\eta_{t\theta_n}(x)\right]\big)f(x/n)\bigg)^2\bigg]
	\le
\sum_{x\in \Z^d} \eta(x)\big(  P^n_t(f^2)-(P^n_t f)^2\big)(x/n)\ .	
\end{equation}
\end{proposition}
\begin{proof}In what follows,   $\eta_0=\eta$.
	Writing $\eta_t(x,y):=\eta_t(x)\left(\eta_t(y)-\IND_x(y)\right)$, we get
	\begin{align}\nonumber
		\E_\eta\bigg[\bigg( \sum_{x\in \Z^d} \big(\eta_{t\theta_n}(x)&-\E_{\eta_0}\left[\eta_{t\theta_n}(x)\right]\big)f(x/n)\bigg)^2\bigg]
	\\ &=\sum_{x,y \in \Z^d}f(x/n)f(y/n)\left(\E_\eta\left[\eta_{t\theta_n}(x,y)\right] -\E_\eta\left[\eta_{t \theta_n}(x)\right]\E_\eta\left[\eta_{t\theta_n}(y)\right] \right)\\
	&\qquad + \sum_{x\in \Z^d} f(x/n)^2\, \E_\eta\left[\eta_{t\theta_n}(x)\right]\ .
	\label{eq:var-estimate-step1}
	\end{align}
	Using \eqref{eq:duality-rel1} and the symmetry of $P_t^n$ in $L^2(\cW^n)$,  the second term on the right-hand side of \eqref{eq:var-estimate-step1} equals 	 the term $\sum_{x\in \Z^d} \eta(x)\, P_t^n(f^2)(x/n)$ on the right-hand side of \eqref{eq:negative-dependence}. By \eqref{eq:duality-rel1}--\eqref{eq:duality-rel2} and the symmetry of $P_t$ and $P^{(2)}_t$ with respect to the measure $(\alpha_x)_{x\in \Z^d}$ and $(\alpha(x,y))_{x,y\in \Z^d}:= (\alpha_x\left(\alpha_y-\IND_y(x)\right))_{x,y\in \Z^d}$ on $\Z^d$ and $\Z^d\times \Z^d$, respectively, the first term on the right-hand side of \eqref{eq:var-estimate-step1} reads as follows: 
	\begin{equation}\label{eq:100}
	\begin{aligned}		
		&\sum_{x,y \in \Z^d} f(x/n)f(y/n)\left(\alpha(x,y)\, P_{t\theta_n}^{(2)}D(x,y)-\alpha(x)\alpha(y)\, P_{t\theta_n}D(x) P_{t\theta_n}D(y)\right)\\
		 &=\sum_{x,y\in \Z^d} (P^{(2),n}_t-(P^n_t\otimes P^n_t))(f\otimes f)(x/n,y/n)\, \eta_0(x,y) -\sum_{x \in \Z^d} (P^n_t f(x/n))^2\, \eta_0(x)\ .
	\end{aligned}
	\end{equation}
	Note that the second term on the right-hand side of \eqref{eq:100} is $-\sum_{x\in \Z^d}\eta(x)\,(P^n_t f)^2(x/n)$, i.e., the remaining term on the right-hand of \eqref{eq:negative-dependence}. We conclude by showing that the first term on the right-hand side of \eqref{eq:100} is non-positive, thus, producing the inequality in \eqref{eq:negative-dependence}; for notational convenience, we only show this for $n=1$. Letting $A\oplus A$ and $A^{(2)}$ denote the Markov generators of $P_t\otimes P_t$ and $P^{(2)}_t$, respectively,  by integration-by-parts (arguing as in, e.g., \cite[Ch.\ VIII, Prop.\ 1.7]{liggett_interacting_2005-1}), we get	
	\begin{align*}
	(P^{(2)}_t-(P_t\otimes P_t))&(f\otimes f)(x,y)
	=\int_0^t P^{(2)}_{t-s}\left(A^{(2)}-A\oplus A\right)(P_sf\otimes P_sf)(x,y)\, \dd s\\
	&= \sum_{z,w\in \Z^d} \int_0^t P^{(2)}_{t-s}\IND_{(z,w)}(x,y)\, \left(A^{(2)}-A\oplus A\right) \left(P_sf\otimes P_s f\right)(z,w)\, \dd s\\
	&= -\theta_{n=1}\sum_{z\sim w}\alpha_z^{a-1}\alpha_w^{a-1}\int_0^t \left(P^{(2)}_{t-s}\IND_{\{z,w\}}(x,y)\right) \left(P_s f(z)-P_s f(w)\right)^2\dd s\\
	&\le 0\ ,
	\end{align*}
where the second-to-last step follows by the explicit expressions of $A^{(2)}$ and $A\oplus A$.
	Since $\eta_0(x,y)\ge 0$, we get the desired claim: \begin{equation*}\sum_{x,y\in \Z^d} \big(P^{(2)}_t-P_t\otimes P_t\big)(f\otimes f)(x,y)\, \eta_0(x,y)\le 0\ .
	\end{equation*} This concludes the proof.	
\end{proof}

\subsubsection{Conclusion of {\normalfont{Part I}}}\label{sec:I-II-III-X}
We now prove that ${\rm I}_n, {\rm II}_n, {\rm III}_n, {\rm I}_n^*, {\rm II}_n^*,{\rm III}_n^*\Longrightarrow 0$ as $n\to \infty$. 

\medskip
\noindent
\textit{${\rm I}_n, {\rm I}_n^*\Longrightarrow 0$ (dynamic noise).} Fix $n\in \N$. Then, $\Q$-a.s., by  Proposition \ref{pr:variance}, we get
\begin{equation}\label{eq:ub-variance-In}
\E_{\mu_n}\left[\left({\rm I}_n \right)^2 \right]\le \frac{1}{n^{d/\beta}}\left\{\frac{1}{n^{d/\beta}}\sum_{x\in \Z^d} \E_{\mu_n}\left[\eta_0(x)\right] P^n_t(f^2)(x/n)\right\}
\le \frac{1}{n^{d/\beta}}\, \langle \cW^n| f^2\rangle\ .
\end{equation}
Since $\langle \cW^n|f^2\rangle \Longrightarrow \langle\cW|f^2\rangle$ by \eqref{eq:conv-W} and $n^{-d/\beta}\to 0$,  ${\rm I}_n$ vanishes in distribution.
 By \eqref{eq:ub-variance-In}, we also  obtain
\begin{equation}
\E_{\mu_n}\left[\left({\rm I}_n^*\right)^2\right]\le \frac{n^{2d(1/\beta-1)}}{n^{d/\beta}}\langle \cW^n|(f^\alpha_n)^2\rangle=\frac{1}{n^{2d}}\sum_{x\in \Z^d}\frac{f(x/n)^2}{\alpha_x}\ ,
\end{equation}
and, since $\alpha_x\ge 1$ for $\Q$-a.e.\ $\alpha$, the right-hand side vanishes as $n\to \infty$.

\medskip
\noindent
\textit{${\rm II}_n, {\rm II}_n^*\Longrightarrow 0$ (initial conditions' variance).} Fix $n \in \N$. Then, for $\Q$-a.e.\ $(\alpha_x)_{x\in \Z^d}$, we have, by symmetry of $P^n_t$ with respect to $(\alpha_x)_{x\in \Z^d}$ and the first-order duality relation	 \eqref{eq:duality-rel1}, 
\begin{equation*}
{\rm II}_n= \langle \cW^n| \left(D^n_0-\rho^n_0\right) P^n_t f\rangle \ .
\end{equation*}
The explicit form of $\mu_n$ (see \eqref{eq:product-local-eq}) yields, $\Q$-a.s., 
\begin{align}\label{eq:prod-measure-use}
\E_{\mu_n}\left[\left({\rm II}_n\right)^2\right]\le \frac{1}{n^{d/\beta}}\left\|P^n_t f\right\|_{n,2}^2	\le \frac{1}{n^{d/\beta}} \left\|f\right\|_{n,2}^2\ ,
\end{align}
where for the last step we used  $\left\|P^n_t f\right\|_{n,2}\le \left\| f\right\|_{n,2}$ (an immediate consequence of  Jensen inequality and the symmetry of $P_t^n$ on $L^2(\cW^n)$).	
Since $f\in \cC_c^+(\R^d)$, \eqref{eq:conv-W}  ensures that $ \left\|f\right\|_{n,2}^2\Longrightarrow \langle \cW|\left|f\right|^2\rangle$.  Therefore,   ${\rm II}_n$ vanishes in distribution as $n\to \infty$.	
Similarly, we have
\begin{equation}
		\E_{\mu_n}\left[\left({\rm II}_n^*\right)^2\right]\le \frac{n^{2d(1/\beta-1)}}{n^{d/\beta}}\norm{f_n^\alpha}_{2,n}^2 = \frac{1}{n^{2d}}\sum_{x\in \Z^d} \frac{f(x/n)^2}{\alpha_x}\xrightarrow{n\to \infty}0\ ,\qquad \Q\text{-a.s.}\ .
\end{equation}

\medskip
\noindent
\textit{${\rm III}_n, {\rm III}_n^*\Longrightarrow 0$ (initial conditions' mean).}	Fix $n \in \N$. Then,  by duality \eqref{eq:duality-rel1} and symmetry of $P_t^n$ in $L^2(\cW^n)$, we obtain, $\Q$-a.s.,
\begin{equation*}
{\rm III}_n= \langle \cW^n| \left(\rho^n_0-\rho_0\right)P^n_tf\rangle\ .
\end{equation*}
Since $\mu_n=\otimes_{x\in \Z^d}{\rm Bin}(\alpha_x,\rho_0(x/n))$, 	 we have \begin{equation}\label{eq:prod-measure-use-2}\rho_0^n\equiv \rho_0\ ,	
\end{equation} and, thus,  the terms ${\rm III}_n$ and ${\rm III}_n^*$ are both $\Q$-a.s.\  identically equal to zero.

\smallskip
 This concludes  Part I of the proof of Theorem \ref{th:HDL}.

\begin{remark}\label{remark: quenched vanish noise2}
	By the argument above, the three terms ${\rm I}_n^*$, ${\rm II}_n^*$ and ${\rm III}_n^*$ even vanish for $\Q$-a.e.\ realization of the environment.
\end{remark}

\begin{remark}[Product measure as an initial condition]\label{rem:prod-measure}
		The choice of $\mu_n$ in \eqref{eq:product-local-eq} as the starting
		 distribution is not essential for our proof of Theorem \ref{th:HDL} to hold. Indeed, the first inequality in \eqref{eq:prod-measure-use} and the identity in \eqref{eq:prod-measure-use-2}  are the only two instances in which the  form of the initial distribution plays a role. These two properties are clearly not a prerogative of product measures, and other initial measure satisfying them (or, possibly, some quantitatively relaxed variants of them) would equally work. 		
\end{remark}

\subsection{Proof of Theorem \ref{th:HDL}. Part II(a)}\label{sec:proof-th:main-general2a} We now prove that, for $d\ge 2$ and  some $\kappa=\kappa(\Q,\beta,a,d)>0$,   \begin{equation}\label{eq:proof-part-II-a}n^{d(1/\beta-1)}\langle P^n_t\rho_0\cW^n| f^\alpha_n\rangle\underset{n\to \infty}\Longrightarrow\langle \mathcal P_{\kappa t}\rho_0\,{\rm Leb}_{\R^d} |f\rangle\ .
	\end{equation}
 Combined with Part I, \eqref{eq:proof-part-II-a} yields the  claim of Theorem \ref{th:HDL} for the field $\cZ_t^n$ for $d\ge 2$.

Also for this proof, we need a preliminary result. This is the content of the following proposition, which was obtained in \cite[Prop.\ 2.4]{Chiarini_Flo_sau} as a consequence of  the quenched invariance principle for ${\rm BTM}(a)$, $a\in [0,1]$, when starting from the origin (see \cite[Thm.\ 1.3]{BarlowCerny} for $d\ge 3$, and \cite[Thm.\ 1.2]{cernyEJP} for $d=2$) and the ergodic theorem.

\begin{proposition}[{\cite[Prop.\ 2.4]{Chiarini_Flo_sau}}]\label{proposition: l1 quenched conv sem} There exists $\kappa=\kappa(\Q,\beta,a,d)>0$ satisfying, for $\Q$-a.e.\ $\alpha$ and for all $t>0$,  compact sets $K\subset \R^d$, and uniformly continuous  bounded functions $g:\R^d\to \R$,
\begin{align}
\lim_{n\to \infty}\frac{1}{n^d}\sum_{\substack{x\in \Z^d\\
		x/n\in K}} \left| P^n_t g(x/n)- \cP_{\kappa t} g(x/n) \right|=0\ .
\end{align}
\end{proposition}

	In conclusion, 
 the last term in \eqref{eq:decomposition-Z} reads as
\begin{align*}
n^{d(1/\beta-1)}\langle P^n_t\rho_0\cW^n| f^\alpha_n\rangle&= \frac{1}{n^d}\sum_{x\in \Z^d} P^n_t \rho_0(x/n)\, f(x/n)\ ,
\end{align*}
which clearly converges, by Proposition \ref{proposition: l1 quenched conv sem},  the continuity of $\cP_{\kappa t}\rho_0$, and $f\in \cC_c^+(\R^d)$, to the desired quantity:
\begin{align}\label{eq:last-step}
\frac{1}{n^d}\sum_{x\in \Z^d} P^n_t \rho_0(x/n)\, f(x/n)\underset{n\to \infty}\Longrightarrow \int_{\R^d} \cP_{\kappa t} \rho_0(x)\, f(x)\, \dd x\ .
\end{align}
This concludes the proof of \textbf{Part II(a)}. 

\begin{remark}[Quenched results]\label{remark: quenched vanish noise}
	Remark \ref{remark: quenched vanish noise2} and Proposition \ref{proposition: l1 quenched conv sem} ensure that the convergence of $\cZ^n_t$ in Theorem~\ref{th:HDL} can be upgraded to a quenched one, ensuring the validity of the claim in \eqref{eq:quenched-result}. 
\end{remark}

\subsection{Proof of Theorem \ref{th:HDL}. Part II(b)}\label{sec:proof-th:main-general2b}
This section is devoted to the proof of the following result.
\begin{proposition} Let $d\ge 2$, 
	 $\beta\in (0,1)$ and $a\in [0,1]$. Then, for all $t\ge 0$ and $f\in \cC_c^+(\R^d)$,
	\begin{align}\label{eq:correct1}
	\frac{1}{n^{d/\beta}}\sum_{x\in \Z^d} f(x/n)\, P^n_t\rho_0(x/n)\, \alpha_x \underset{n\to \infty}\Longrightarrow \langle \cW|f \rho_0\rangle\ .
	\end{align}
\end{proposition}
\begin{proof}
	The first step consists in showing that, for what concerns the summation on the left-hand side of \eqref{eq:correct1}, the terms which contribute to the limit are only those with weights $\alpha_x$'s large as $(1+o(1))n^{d/\beta}$. In other words, terms corresponding to	 strictly smaller weights do not contribute to the limit.	
	
	The proof of this fact is rather standard and goes as follows. Consider a sequence $(r_n)_{n\in \N}$ satisfying $r_n\to \infty$ and $r_n=o(n^{d/\beta})$.  For the next step, we shall further require that  $\theta_n=o(r_n)$.	Then, for $\Q$-a.e.\ $\alpha$ and all $n\in \N$,  we split the summation as follows:
	\begin{align}\nonumber
	\frac1{n^{d/\beta}}\sum_{x\in \Z^d} f(x/n)\, P^n_t\rho_0(x/n)\, \alpha_x&= \frac1{n^{d/\beta}}\sum_{x\in \Z^d} f(x/n)\, P^n_t\rho_0(x/n)\, \alpha_x\, \IND_{[r_n,\infty)}(\alpha_x)\\
	\nonumber
	&\quad+
	\frac1{n^{d/\beta}}\sum_{x\in \Z^d} f(x/n)\, P^n_t\rho_0(x/n)\, \alpha_x\, \IND_{[0,r_n)}(\alpha_x)\\
	\label{eq:JK}
	&=: \cJ^n_t+\cK^n_t\ .	
	\end{align} 
	We claim that $\cK^n_t\ge 0$ goes to zero in $\Q$-probability, as $n\to \infty$. Indeed, for every $\varepsilon > 0$, Markov inequality yields 
	\begin{align}\label{eq:markov1}
	\Q\big(\cK^n_t>\varepsilon\big)\le \frac{\E_\Q\left[\cK_t^n\right]}{\varepsilon}	 \le \frac{C(f)}{\varepsilon} \frac{n^d}{n^{d/\beta}}\,\E_\Q[\alpha_0\, \IND_{[0,r_n)}(\alpha_0)]\ ,
	\end{align}
	where the last step is a simple consequence of non-negativity of $f, \rho_0$, the fact that $f$ has compact support and is bounded, and that $P^n_t \rho_0\le 1$. It remains to estimate the expectation on the right-hand side of \eqref{eq:markov1}: for some constants $c,c'>0$, 
	\begin{align}
	\E_\Q[\alpha_0\, \IND_{[0,r_n)}(\alpha_0)] &= \int_0^{r_n}\Q\big(\alpha_0>u\big){\rm d}u \\&=1+\int_1^{r_n}\Q\big(\alpha_0>u\big){\rm d}u
	\le1+ c\int_1^{r_n} u^{-\beta}\,{\rm d}u
	 \le c'\, {r_n}^{1-\beta}\ .
	\end{align}
	Combining this estimate with the one in \eqref{eq:markov1}, we obtain, for every $\varepsilon > 0$,
	\begin{align*}
	\Q\big(\cK^n_t>\varepsilon\big) \le \frac{C'}{\varepsilon} n^{d-d/\beta}r_n^{1-\beta} = \frac{C'}{\varepsilon} \big(n^{-d/\beta}r_n\big)^{1-\beta} = o(1)\ ,\qquad \text{as}\ n\to \infty\ ,
	\end{align*} where the last step follows by the assumption $r_n=o(n^{d/\beta})$ and $\beta<1$.	
	
	As a consequence of the above step, the limit in \eqref{eq:correct1} is determined by that of $\cJ^n_t$, which we recall to be given by
	\begin{equation*}
	\cJ^n_t=	\frac1{n^{d/\beta}}\sum_{x\in \Z^d} f(x/n)\, P^n_t\rho_0(x/n)\, \alpha_x\, \IND_{[r_n,\infty)}(\alpha_x)\ .	
	\end{equation*}
	The argument for the symmetric case ($a=0$) is the easiest one. We start by observing that,  $\Q$-a.s., if $x\in \Z^d$ is such that $\alpha_x\ge r_n$, then, since $\rho_0\in [0,1]$,
	\begin{align*}
	\left|P^n_t\rho_0(x/n)-\rho_0(x/n)\right| &\le \mathbf P^\alpha_x\big(X_{t\theta_n}\neq x\big)\\ &\le 1-\mathbf P^\alpha_x\big(X_s =x\,,\, \forall\,s \in [0,t\theta_n] \big)
	= 1- e^{-t\theta_n/r_n} \xrightarrow{n\to \infty}0\ ,
	\end{align*}
	where for the last step we used $\theta_n=o(r_n)$.
	Henceforth, 
	\begin{align}
	\cJ^n_t= \frac{1}{n^{d/\beta}}\sum_{x\in \Z^d}f(x/n)\, \rho_0(x/n)\, \alpha_x\, \IND_{[r_n,\infty)}(\alpha_x) + o_\Q(1) = \cJ^n_0+o_\Q(1)\ ,
	\end{align}
	where $o_\Q(1)$ denotes a random variable which vanishes in $\Q$-probability as $n\to \infty$. This completes the proof of the proposition in the case $a=0$.

	As for the asymmetric case $a\in (0,1]$, we need to complement the information that $\alpha_x$ is large with the information that, with  $\Q$-probability going to one as $n\to \infty$, close to $x$, we observe only small weights.
	For this purpose, fix $a\in (0,1]$, and define 
	\begin{equation}\label{eq:rn}
	r_n = n^{d/\beta} \log(1+n)^{\frac14\left(1-1/\beta\right)}\ ,
\end{equation}
	which, recalling the definition of $\theta_n$ in \eqref{eq:theta-n} and that $\beta\in (0,1)$, clearly satisfies $r_n\to \infty$, $r_n=o(n^{d/\beta})$ and $\theta_n = o(r_n)$  for all $d\geq2 $. 
	We further introduce a second sequence $s_n\to \infty$, to be  defined as
	\begin{equation}\label{eq:sn-kn}
		s_n = \log(1+n)^{-\frac13\left(1-1/\beta\right)}\ .	
	\end{equation}
With this choice, note that, for all $d\ge 2$,	
\begin{equation}\label{eq:sn-bounds}
	n^{d/\beta}= o(r_ns_n)\ ,
\end{equation}
and
\begin{equation}\label{eq:sn-bounds2}
 \frac{\theta_n s_n^2}{r_n} = O(\log(1+n)^{\frac1{12}\left(1-1/\beta\right)})=o(1)\ 	 .
\end{equation}
Now, define, for all $\alpha$,
	\begin{equation*}
		R_n=R_n(\alpha,f):=\{x\in \Z^d\mid x/n \in {\rm supp}(f)\,,\, \alpha_x\ge r_n\}\subset \Z^d\ .
	\end{equation*}
	We wish to show that all weights at a finite distance from $R_n$ are larger than $s_n$ with vanishing $\bbQ$-probability. With this in mind, we introduce 
	\begin{equation*}
	\cS_n=\cS_n(\alpha,f):= \left\{\max\{\alpha_y\mid 	y \in \Z^d\setminus R_n\,,\, {\rm dist}(y,R_n)\le 2\}> s_n\right\}\ ,
	\end{equation*}
	and we see, by a union bound,  that 
	\begin{align}\label{eq:cSn}
	\Q\big(\cS_n\big)\le C(f)\,n^d\, \Q\big(\alpha_0\ge r_n\,,\, \alpha_y>s_n\ \text{for some}\ 1\le |y|\le 2\big)\le C'\frac{n^d}{r_n^\beta s_n^\beta} =o(1)\ ,
	\end{align} 
where the last inequality is a consequence  of our hypotheses on the $\alpha$'s (i.e., independence and \eqref{eq:heavy-tail}), whereas the last identity follows from \eqref{eq:sn-bounds}.
	We can now conclude. Indeed, for $\alpha\in  \cS_n^c$, letting
	$x\in R_n$ and \begin{equation*}U(x):=\{y\in \Z^d\mid 1\le |y-x|\le 2\}\ ,
		\end{equation*}
	we have that, for some  constants $c_1,c_2,c_3>0$ (independent of $\alpha$):
	\begin{enumerate}
		\item[(i)] the jumps from $x$ to $U(x)$ are  stochastically dominated  by an exponential random variable of rate $c_1\, r_n^{a-1}s_n^a$;	
		\item[(ii)] the probability that, right after exiting $x$, the next jump is not back to $x$ is smaller than $\frac{c_2\,s_n^a}{c_2\,s_n^a+r_n^a}\le c_3 (s_n/r_n)^a$.
	\end{enumerate}	
Therefore, for all $\alpha\in \cS_n^c$,   claim (ii) ensures that the number of attempts required to perform a jump from a neighbor of $x\in R_n$ to another vertex in $U(x)$ (thus, different from $x\in R_n$) is stochastically dominated by a geometric random variable of success probability $c_3(s_n/r_n)^a$. Moreover, for all $\alpha \in \cS_n^c$,  claim  (i) ensures
that the number of times that a random walk started at $x\in R_n$ exits $x$ (thus, visiting a neighbor in $U(x)$) in the interval $[0,t\theta_n]$ is stochastically dominated by a Poisson random variable of parameter $c_1t\theta_n r_n^{a-1}s_n^a$. Since these two mechanisms are independent, we may apply the Markov inequality so to obtain, for some $C=C(t,c_1,c_3)>0$ and all $\alpha \in \cS_n^c$, 
\begin{equation}
	 \max_{x \in R_n}\mathbf P^\alpha_x\bigg(\sup_{s\in [0,t\theta_n]}|X_s-x|> 2\bigg)\le C\, \theta_n \frac{s_n^{2a}}{r_n}\le C\, \theta_n\frac{s_n^2}{r_n}\xrightarrow{n\to \infty} 0\  ,
\end{equation} 
where the last step is a consequence of \eqref{eq:sn-bounds2}. Now, recall that $\rho_0$ is uniformly continuous and bounded $\rho_0\in [0,1]$. Further, let $w_{\rho_0}(\delta)$ denote the $\delta$-modulus of continuity of $\rho_0$. Then, by the  aforementioned properties of $\rho_0$ and the above display, we get, uniformly over $\alpha \in \cS_n^c$, 	
\begin{equation}
	\begin{aligned}
	\Psi_n:=	\max_{x\in R_n}|P^n_t \rho_0(x/n) -\rho_0(x/n)| &\leq \max_{x\in R_n} \bfE^\alpha_x[|\rho_0(X_{t\theta_n}/n) - \rho_0(x/n)|]\\
		& \leq \max_{x\in R_n} \bfP^\alpha_x\bigg(\sup_{s\in [0,t\theta_n]}|X_s-x|>2\bigg) + w_{\rho_0}(2/n)\xrightarrow{n\to \infty}0\ .
	\end{aligned}
\end{equation}
  After observing that
	\begin{align*}
	\abs{\cJ^n_t-\cJ^n_0}&\le \frac{1}{n^{d/\beta}}\sum_{x\in \Z^d}f(x/n)\,|P^n_t \rho_0(x/n)-\rho_0(x/n)|\,\alpha_x\, \IND_{[r_n,\infty)}(\alpha_x)\\
	&\le \Psi_n\,
	\frac{1}{n^{d/\beta}}\sum_{x\in \Z^d}f(x/n)\,  \alpha_x = \Psi_n\, \langle \cW^n|f\rangle\ ,
	\end{align*}
	we get, for all $\varepsilon, \delta> 0$, 
	\begin{align*}
	\Q\big(|\cJ^n_t-\cJ^n_0|>\varepsilon)&\le \Q\big(\{\Psi_n\langle \cW^n|f\rangle>\varepsilon\}\cap \cS_n^c\big)+\Q\big(\cS_n\big)\\
	&=\Q\big(\{\Psi_n>\delta\}\cap \cS_n^c\big)+\Q\big(\langle \cW^n|f\rangle >\varepsilon/\delta\big)+\Q\big(\cS_n\big)\ .
	\end{align*}
For every $\varepsilon > 0$, the third term on the right-hand side above vanishes, for all $\delta > 0$, as $n\to \infty$, by \eqref{eq:cSn}; the second term vanishes by the convergence in \eqref{eq:conv-W} and by first taking $n\to \infty$ and then $\delta \to 0$; the first term vanishes, for all $\delta > 0$, as $n\to \infty$, 
 because $\lim_{n\to \infty}\sup_{\alpha \in \cS_n^c}\Psi_n= 0$. This concludes the proof of the proposition.
\end{proof}

\subsection{Proof of Theorem \ref{th:HDL}. Part II(c)}\label{sec:proof-th:main-general2c}

Recall that, in this section, we deal with the one-dimensional case. Here, we provide an alternative construction of the underlying randomness, which turns out to be useful for what follows. Here, we write $\overset{{\rm Law}}=$ for equality in law.

\begin{proposition}\label{pr:coupling} Let $d=1$. There exists a coupling $(\bar \cW, (\bar{ \mathcal P}_t)_{t\ge 0}, (\bar\cW^n)_n, ((\bar P^n_t)_{t\ge 0})_n)$ with law $\bar \Q$ such that:
	\begin{enumerate}
		\item \label{it:coupling1} 
		$\bar \cW\overset{{\rm Law}}=\cW$, 	$\bar \cW^n\overset{{\rm Law}}=\cW^n$, 	$(\bar \cP_t)_{t\ge 0}\overset{{\rm Law}}=(\cP_t)_{t\ge 0}$, 	$(\bar P_t^n)_{t\ge 0}\overset{{\rm Law}}=(P_t^n)_{t\ge 0}$;  in particular,  the following analogue of \eqref{eq:conv-W}  holds:
		\begin{equation}\label{eq:conv-W-bar}
		\langle \bar \cW^n| f\rangle \underset{n\to \infty}\Longrightarrow \langle \bar \cW|f \rangle\ ,\qquad f \in \cC_c^+(\R)\ ;	
		\end{equation}
		\item \label{it:coupling2} for all $g\in \cC_b(\R)$,   $\bar {\mathcal P}_tg\bar \cW  \overset{{\rm Law}}=\mathcal P_tg\cW  $, and similarly $\bar P^n_tg\bar \cW^n\overset{{\rm Law}}=P^n_tg\cW  $;
		\item \label{it:coupling1.5} $\bar \Q$-a.s.,  for all $g\in \cC_b(\R)$ and $t\ge 0$, $\bar \cP_t g:\R\to \R$ is continuous and bounded;
		\item \label{it:coupling4} $\bar \Q$-a.s., for all $g\in \cC_b(\R)$ and compact sets $K\subset \R$, $\left\|\IND_K\left(\bar P^n_t g- \bar{\mathcal P}_{\kappa t} g\right)\right\|_{n,\infty}\to 0$.
		\item \label{it:coupling3} $\bar \Q$-a.s., $\bar \cW^n\to \bar \cW$ in $\cM_v(\R)$.
	\end{enumerate}
\end{proposition}

\begin{proof}
One possible coupling is the one given in \cite[\S\S1--6]{BenArousCerny05}. Namely, first  introduce the random locally finite measure $\bar \cW$ (\cite[p.\ 1172]{BenArousCerny05}). Conditional on $\bar \cW$, define the Markov process $(\bar Z_t(x))_{t\ge 0,x\in \R}$, with $\bar Z_0(x)=x$, as in \cite[Def.\ 1.1]{BenArousCerny05} (note that there is no problem in changing the starting point of the process, see, e.g., the end of \cite[p.\ 1167]{BenArousCerny05}); to such a Markov process, associate its Markov semigroup $\bar{\mathcal P}_t$. By item (iv) in \cite[Prop.\ 3.2]{ben-arous_cerny_dynamics_2006}, $\bar {\mathcal P}_t\rho_0:{\rm supp}(\bar \cW)\subset \R\to \R$ admits a unique continuous and bounded extension (see also \cite{croydon2019heat}), which, by a slight abuse of notation, we keep calling $\bar{\mathcal P}_t$; this yields item \eqref{it:coupling1.5}. From the randomness used to generate $\bar \cW$, construct, for every $n \in \N$, the random speed measure $\bar \cW^n$, the random scale function $\bar S_n$,  and the semigroups $\bar P^n_t$ and $\bar P^{n,\bar S_n}_t$ corresponding to the Markov jump processes evolving on $\Z_n:=\frac{\Z}{n}$ and on the tilted support $\bar S_n(\Z_n)\subset \R$ as in \cite[p.\ 1172--3]{BenArousCerny05}. Such definitions immediately yield items \eqref{it:coupling1}, \eqref{it:coupling2}, and \eqref{it:coupling3} of the proposition. Finally, by 
	the first claim in  \cite[Prop.\ 6.2]{BenArousCerny05}  --- which generalizes as, e.g., in item (iv) of \cite[Prop.\ 2.5]{BenArousCerny05},
	%
	%
	we get: conditional on $\bar \cW$, for all $t >0$, $x \in \R$, and for all $(x_n)_n$, $x_n \in \Z$, and $x_n/n\to x$, 
	\begin{equation}\label{eq:conv-1}
	\bar P^n_t g(x_n/n)\xrightarrow{n\to \infty}\bar {\mathcal P}_{\kappa t} g(x)\ ,\qquad \kappa=1\ .
	\end{equation}
	By the space-continuity of $\bar {\mathcal P}_{\kappa t} g$, \eqref{eq:conv-1} is actually equivalent to item \eqref{it:coupling4} of the proposition, thus, concluding the proof.
\end{proof}
We are now ready to conclude the proof of this step of the proof of Theorem \ref{th:HDL}. Indeed, the following convergence in distribution
\begin{align}n^{1/\beta-1}\langle P^n_t\rho_0\cW^n| f^\alpha_n\rangle
 = \frac1n\sum_{x\in \Z} P_t^n\rho_0(x/n)\, f(x/n)\underset{n\to \infty}\Longrightarrow \int_\R \cP_{\kappa t}\rho_0(x)\, f(x)\, \dd x = \langle \cP_{\kappa t}\rho_0\,{\rm Leb}_{\R}|f\rangle
	\end{align} readily follows from the fact that $f\in \cC_c^+(\R)$ and the above proposition. This proves the desired claim involving the field $\cZ_t^n$. As for the density field $\cX_t^n$, it suffices to show that
\begin{equation}\label{eq:conv-bar}
P^n_t \rho_0\cW^n  \underset{n\to \infty}\Longrightarrow {\mathcal P}_{\kappa t}\rho_0\cW  \ \quad \text{in}\ \cM_v(\R)\ ,
\end{equation}
which,  in view of Proposition \ref{pr:coupling}\eqref{it:coupling2}, 
is  a consequence of the following result.
\begin{lemma} For $d= 1$, with the same notation as in Proposition \ref{pr:coupling}, the following convergence 
	$
	\bar P^n_t \rho_0	\bar \cW^n  \underset{n\to \infty}\Longrightarrow  \bar{\mathcal  P}_{\kappa t} \rho_0\bar \cW $ in $\cM_v(\R)$
 holds  $\bar \Q$-a.s.
\end{lemma}

\begin{proof}
We divide the proof of the lemma into steps.
	\begin{enumerate}
		\item By the triangle inequality, we get
		\begin{align*}
		&	\left|	\langle \bar P^n_t \rho_0\bar \cW^n | f\rangle-\langle  \bar {\cP}_{\kappa t} \rho_0\bar \cW|f\rangle\right|\le \left|
		\langle \left(\bar P^n_t \rho_0-\bar {\cP}_{\kappa t} \rho_0\right)\bar \cW^n| f\rangle
		\right|+\left|
		\langle \bar {\cP}_{\kappa t} \rho_0\bar\cW^n|f\rangle - \langle \bar {\cP}_{\kappa t}\rho_0\bar \cW|f\rangle
		\right|\ .	
		\end{align*}
		\item For every compact subset  $K\supseteq {\rm supp}(f)$ of $\R$, 	   $\bar\Q$-a.s.,  
		\begin{align*}
		\left|\langle \left(\bar P^n_t\rho_0-\bar {\cP}_{\kappa t}\rho_0\right)\bar \cW^n|f\rangle\right|\le 	\left\| \IND_K\left(\bar P^n_t \rho_0-\bar {\cP}_{\kappa t} \rho_0\right) \right\|_{n,\infty}\langle \bar \cW^n|\left|f\right|\rangle\ .	
		\end{align*}
		By the assumed continuity of $\rho_0$, we may apply item \eqref{it:coupling4} in Proposition \ref{pr:coupling}, yielding, together with \eqref{eq:conv-W-bar}, $\bar \Q$-a.s., $	\langle \left(\bar P^n_t \rho_0-\bar {\cP}_{\kappa t} \rho_0\right)\bar \cW^n| f\rangle\to 0$.
		\item Since, by item \eqref{it:coupling1.5} in Proposition \ref{pr:coupling}, $\bar \Q$-a.s., $\bar {\cP}_{\kappa t} \rho_0 \cdot f \in \cC_c^+(\R)$ for all $f \in \cC_c^+(\R)$, item \eqref{it:coupling3} in Proposition \ref{pr:coupling} implies, $\bar \Q$-a.s., 
		\begin{equation}
		\langle \bar {\cP}_{\kappa t} \rho_0\bar \cW^n | f\rangle \underset{n\to \infty}\longrightarrow \langle  \bar {\cP}_{\kappa t}\rho_0\bar \cW|f\rangle\ ,\qquad f \in \cC_c^+(\R)\ .
		\end{equation}

	\end{enumerate}
	This concludes the proof of the lemma.
\end{proof}

\begin{appendix}
    \section{Construction of the infinite particle system}\label{section: construction}
    In this section, we provide a construction of the infinite particle system of interacting ${\rm BTM}(a)$ solely based on the self-duality and self-intertwining relations satisfied by the system of finitely many interacting ${\rm BTM}(a)$ (see, e.g., \cite[\S 4.1]{FloreaniJansenRedigWagner}).
    
    \label{sec:appendix-second-construction}
    More precisely, let us prove that, for $\Q$-a.e.\ environment $\alpha$ and for all $\eta \in \Xi=\Xi_\alpha$, there exists a $\Xi$-valued Markov process $(\eta_t)_{t\ge0}$ with marginal distributions being the infinite-particle analogues of those of the finite particle system.

    All throughout, fix an environment $\alpha \in \N^{\Z^d}$.
    \begin{itemize}
        \item The state space $\Xi=\Xi_\alpha:=\prod_{x\in\Z^d}\{0,1,\ldots, \alpha_x\}$ endowed with the product topology is Hausdorff and compact  (as product of the Hausdorff and compact spaces). We equip $\Xi$ with the product Borel $\sigma$-algebra.
        \item Let $\cT:=[0,T]$, and for every non-empty $\cS\subset \cT$ we consider $\Xi^\cS:= \prod_{s \in \cS}\Xi$ endowed with the product $\sigma$-algebra $\otimes_{s\in \cS}\,\cB_\Xi$.
    \end{itemize}
    The idea is that of applying Kolmogorov extension theorem in space first, and then in time. For this purpose, we need the finite-particle dual process' existence as an input: for every $k \in \N$, 
    \begin{equation}
    p^{(k)}_t(\mathbf x,\mathbf y):= \mathbf P_{\mathbf x}\left(\mathbf X_t=\mathbf y\right)\ ,\qquad t \ge 0\ ,\ \mathbf x, \mathbf y \in (\Z^d)^k\ ,
    \end{equation}
    denotes the transition probability of $k$-stirring labeled particles in the ladder-environment $\alpha$, as defined, e.g., in \cite[App.\ A]{floreani_HDL_2021}  (see also \S\ref{sec:appendix-first-construction} below for this ladder construction). Because of the well-posedness of the single-particle dynamics (i.e., the random walk) from \textit{all} initial positions, this stirring dynamics  with finitely-many particles is likewise well-defined. It is simple to check that, for $k\in \N$ and   $\mathbf x=(x_1,\ldots, x_k)\in (\Z^d)^k$,
    \begin{align}
    \alpha(\mathbf x):= \alpha_{x_1}\left(\alpha_{x_2}-\IND_{x_1}(x_2)\right)\cdots \left(\alpha_{x_k}-\sum_{j=1}^{k-1}\IND_{x_j}(x_k)\right)\ , 
    \end{align}
    is the (unnormalized) reversible measure for the $k$-particle process $\mathbf X_t$.		
    For later convenience,   we also define the multivariate  falling factorials of $\eta\in \Xi$: 
    \begin{align}
    \eta(\mathbf x):= \eta(x_1)\left(\eta(x_2)-\IND_{x_1}(x_2)\right)\cdots \left(\eta(x_k)-\sum_{j=1}^{k-1}\IND_{x_j}(x_k)\right)\ .
    \end{align}

    \subsubsection*{Construction of one-time distributions}
    Fix $\eta \in \Xi$ and $t \in [0,T]$, and define the $\Xi$-valued random element $\eta_t$ as follows:
    \begin{itemize}
        \item For every finite $\Lambda = \{x_1,\ldots, x_\ell\} \subset \Z^d$, let $\eta_t^\Lambda:=(\eta_t(x_1),\ldots, \eta_t(x_\ell))$ be the $\ell$-dimensional random vector, whose law is denoted by $\pi^\Lambda_{\eta,t}$ and satisfies (see, e.g., \cite[\S 4.1]{FloreaniJansenRedigWagner})
        \begin{align}\label{eq:falling-factorial-moments}
        \int_{\prod_{x\in \Lambda}\{0,1,\ldots, \alpha_x\}} \eta^\Lambda(\mathbf x)\, \pi^\Lambda_{\eta,t}(\dd \eta^\Lambda) 
        &= \sum_{\mathbf y\in (\Z^d)^k} \eta(\mathbf y)\, p^{(k)}_t(\mathbf y,\mathbf x)
        \end{align}
        for all $\mathbf x \in \Lambda^k$, $k \in \N$. Note that:
        \begin{itemize}
            \item the right-hand side of \eqref{eq:falling-factorial-moments} is non-negative and finite since $\alpha(\mathbf x)\in [0,\infty)$,  $\alpha(\mathbf x)\, p_t^{(k)}(\mathbf x,\mathbf y)= \alpha(\mathbf y)\, p_t^{(k)}(\mathbf y,\mathbf x)$,  $\frac{\eta}{\alpha}\in [0,1]$, and $\sum_{\mathbf y \in (\Z^d)^k}p^{(k)}_t(\mathbf x,\mathbf y)=1$; 
            \item since $\prod_{x\in \Lambda}\{0,1,\ldots, \alpha_x\}$ is finite, the law $\pi^\eta_{t,\Lambda}$ is uniquely characterized by all its (factorial) moments in  \eqref{eq:falling-factorial-moments}.
        \end{itemize} 
        \item For each pair of  finite subsets $\Lambda_1\subset \Lambda_2\subset \Z^d $, the laws of $\eta_t^{\Lambda_1}$ and $\eta_t^{\Lambda_2}$ are compatible. Indeed, letting, for all functions $f:\prod_{x\in \Lambda_1}\{0,1,\ldots,\alpha_x\}\to \R$  and for all $\Lambda_1\subset \Lambda_2$,
        \begin{equation*}
        \big(f\circ \varphi_{\Lambda_1,\Lambda_2}\big)(\eta^{\Lambda_2}):= 		f(\eta^{\Lambda_1})\ ,
        \end{equation*}
        compatibility means that 
        \begin{equation}
        \pi^{\Lambda_1}_{\eta,t}\, f = \pi^{\Lambda_2}_{\eta,t}\,(f\circ \varphi_{\Lambda_1,\Lambda_2})\ ,\qquad f:\prod_{x\in \Lambda_1}\{0,1,\ldots, \alpha_x\}\to \R\ .
        \end{equation}
        However, this is clear since moments are measure-determining in this finite state space context (no summability condition is required).
    \end{itemize}
    
    By the Kolmogorov extension theorem (see, e.g., \cite[Thm.\ 7.7.1]{bogachev_measure_2007}, but also the version in \cite{friedli_statistical_2017}, Thm.\ 6.6 and its proof in \S6.12.4, which is carried out precisely for configuration spaces with compact single-spin spaces), there exists a unique law $\pi_{\eta,t}$ on (the product Borel $\sigma$-algebra) $\Xi$ such that its push-forward measure via  the canonical projection from $\Xi$ to $\prod_{x\in \Lambda}\{0,1,\ldots, \alpha_x\}$ coincides with $\pi^\Lambda_{\eta,t}$.

    \subsubsection*{Construction of finite-dimensional distributions}  We define inductively on $n \in \N$ the measures $\pi_{\eta,t_1,\ldots, t_n}$ on the product space $\Xi^n$ as follows:
    \begin{itemize}
        \item For $n=1$, for  all $t\in [0,T]$ and $\eta \in \Xi$, $\pi_{\eta,t}$ was constructed in the previous subsection.
        \item Assume that for  $n\ge 1$, for all $0\le t_1<\ldots<t_{n-1}$ and $\eta \in \Xi$, the measure $\pi_{\eta,t_1,\ldots, t_{n-1}}$  on $\Xi^{n-1}$ is well-defined; then,  we construct $\pi_{\eta,t_1,\ldots, t_n}$ as follows: the random vector
        \begin{align}
        \left(\eta_{t_1},\ldots, \eta_{t_n}\right) \in \Xi^n
        \end{align}
        is  such that:
        \begin{itemize}
            \item the marginal $(\eta_{t_1},\ldots, \eta_{t_{n-1}})\in \Xi^{n-1}$ is distributed as $\pi_{\eta,t_1,\ldots, t_{n-1}}$;
            \item conditionally on $(\eta_{t_1},\ldots, \eta_{t_{n-1}})$, the last marginal $\eta_{t_n} \in \Xi$ is distributed as $\pi_{\eta_{t_{n-1}},t_n-t_{n-1}}$ as constructed in the previous subsection for all initial conditions in $\Xi$ and times.
        \end{itemize}
    \end{itemize}
    This construction yields, for all initial conditions $\eta\in \Xi$, a collection of consistent measures $\{\pi_{\eta,\mathsf t}: \mathsf t=(t_1,\ldots, t_n)\,  ,\ 0\le t_1,\ldots, t_n\le T\, ,\ n \in \N\}$. Since $\Xi$ is compact, for every $t \in [0,T]$, $\pi_{\eta,t}$ has a compact approximating class as in the hypothesis of  \cite[Thm.\ 7.7.1]{bogachev_measure_2007}. Therefore, Kolmogorov's extension theorem therein applies in our context and there exists a unique probability measure $\pi_\eta$ on (the product $\sigma$-algebra of) $\Xi^{[0,T]}$.

    \subsection{Ladder and stirring construction}\label{sec:appendix-first-construction}
        
    For $\Q$-a.e.\ environment $\alpha$, following, e.g., \cite[App.\ A]{floreani_HDL_2021}, 	let us  enlarge the configuration space $\Xi=\Xi_\alpha:=\prod_{x\in \Z^d}\{0,1,\ldots, \alpha_x\}$ by adding so-called \emph{ladders}. This enlargement is the key to reduce the problem of defining a system of  interacting  ${\rm BTM}(a)$ to that for a classical symmetric exclusion process allowing at most one particle per site. 
    
    Consider the following configuration space
    \begin{equation}
        \widetilde \Xi:= \prod_{x\in \Z^d}\prod_{i=1}^{\alpha_x}\{0,1\}\ ,
    \end{equation}
    with $x\in \Z^d$ denoting the original site on $\Z^d$ and $i\in \{1,\ldots, \alpha_x\}$ the $i^{th}$ ladder at $x\in \Z^d$. Letting $S=S^{d,\alpha}:=\{(x,i)\mid  x\in \Z^d\,,\, i=1,\ldots, \alpha_x\}$, observe that $\widetilde \Xi=\{0,1\}^S$. Now, for every parameter  $a\in [0,1]$, we wish to define a symmetric exclusion process $(\zeta_t)_{t\ge 0}$ on $\widetilde \Xi$ with the following jumping mechanism: a particle jumps from $(x,i)\in S$ to $(y,j)\in S$ with rate given by
    \begin{equation}\label{eq:jump-ladder}
        \zeta(x,i)\left(1-\zeta(y,j)\right)\alpha_x^{a-1}\alpha_y^{a-1}	\IND_{|x-y|=1}\ ,
    \end{equation}
    with $|\cdot|$ denoting here the usual Euclidean distance on $\R^d$.
    In particular, note that, besides the exclusion rule (encoded in the term $\zeta(x,i)\left(1-\zeta(y,j)\right)$)  particles jump with a conductance $\alpha_x^{a-1}\alpha_y^{a-1}\IND_{|x-y|=1}$ which does not depend on the ladder labels $i$ and $j$. Hence, it is immediate to see that, by running a single particle on $S$ according to these rules and ignoring the ladder labels, one recovers the ${\rm BTM}(a)$ defined in \S\ref{sec:BTM}.	
    
    Furthermore, since, $\Q$-a.s.,   $\alpha_x$ is finite for all $x\in \Z^d$, the maximal jump rate from every $(x,i)\in S$ is $\Q$-a.s.\ finite:
    \begin{equation}
        \sum_{\substack{y\in \Z^d\\
                |y-x|=1}} \sum_{j=1}^{\alpha_y}\alpha_x^{a-1}\alpha_y^{a-1} \le  \alpha_x\sum_{\substack{y\in \Z^d\\
                |y-x|=1	}}\alpha_y<\infty\ ,\qquad \Q\text{-a.s.}\ .
    \end{equation}
    This and  the well-known fact  that, for $\Q$-a.e.\ $\alpha$, the ${\rm BTM}(a)$ starting from any $x\in \Z^d$ does not explode a.s.\ are the two key ingredients in \cite[Assumption SSEP]{faggionato_graphical_2023} to ensure, by means of a Harris' graphical construction,  existence of a unique Feller process associated to the jump dynamics given in \eqref{eq:jump-ladder}.
    
    Let $(\zeta_t)_{t\ge 0}$ denote such a Feller process evolving on $\widetilde \Xi$, and let $\widetilde \cL:\widetilde \cD\subset \cC(\widetilde \Xi)\to \cC(\widetilde \Xi)$ denote the corresponding infinitesimal generator. As stated in \cite[Proposition 3.4]{faggionato_graphical_2023}, local functions belong to the domain $\widetilde \cD$ of the generator, and, for such a function $\varphi:\widetilde \Xi\to \R$, we have, for all $\zeta \in \widetilde \Xi$,
    \begin{align}
        \widetilde \cL\varphi(\zeta)&= \sum_{(x,i)\in S}\zeta(x,i)\,\alpha_x^{a-1}\sum_{\substack{(y,j)\in S\\
                |y-x|=1}}\alpha_y^{a-1}\left(1-\zeta(y,j)\right)\big(\varphi(\zeta^{(x,i),(y,j)})-\varphi(\zeta)\big)\\
        &= \sum_{x\in \Z^d}\alpha_x^{a-1} \sum_{\substack{y\in \Z^d\\
                |y-x|=1}}\alpha_y^{a-1} \sum_{i=1}^{\alpha_x}\zeta(x,i)\sum_{j=1}^{\alpha_y}\left(1-\zeta(y,j)\right)\big(\varphi(\zeta^{(x,i),(y,j)})-\varphi(\zeta)\big)\ ,
    \end{align}
    with $\zeta^{(x,i),(y,j)}$ being obtained from $\zeta$ by swapping occupation variables at $(x,i)$ and $(y,j)\in S$. This swapping operation may be also represented via a stirring procedure, and the existence and non-explosiveness  of the finite particle system readily follows from here.
    
    We conclude by observing that, by introducing the following fields associated to $\zeta_t$
    \begin{equation}
        \widetilde \cZ_t^n:= \frac1{n^d}\sum_{x\in \Z^d}\left(\frac{1}{\alpha_x}\sum_{i=1}^{\alpha_x}\zeta_{t\theta_n}(x,i)\right) \delta_{x/n}
    \end{equation}
    and
    \begin{equation}
        \widetilde \cX_t^n:=\frac{1}{n^{d/\beta}}\sum_{x\in \Z^d} \left(\sum_{i=1}^{\alpha_x}\zeta_{t\theta_n}(x,i)\right) \delta_{x/n}\ ,
    \end{equation}
    the analogue of Theorem \ref{th:HDL} for these fields holds true.
\end{appendix}
\vspace{1cm}

\noindent\textbf{Acknowledgments.} 
    The authors wish to thank Frank Redig for suggesting the problem, his insights and support.
    S.F.\ thanks Noam Berger, David Croydon, Ben Hambly and Takashi Kumagai for useful and inspiring discussions, as well as Simona Villa for making the picture.

	F.S.\ was partially supported by the Lise Meitner fellowship, Austrian Science Fund (FWF):	M3211.
    S.F.\ acknowledges financial support from the Engineering and Physical Sciences Research Council of the United Kingdom through the EPSRC Early Career Fellowship EP/V027824/1.
     A.C., S.F.\ and F.S.\ thank the  Hausdorff Institute for Mathematics (Bonn) for its hospitality during the Junior Trimester Program \textit{Stochastic modelling in life sciences} funded by the Deutsche Forschungsgemeinschaft (DFG, German Research Foundation) under Germany’s Excellence Strategy - EXC-2047/1 - 390685813. While this work was written, A.C.~was associated to INdAM (Istituto Nazionale di Alta Matematica ``Francesco Severi'') and the group GNAMPA.


\begin{thebibliography}{10}
	
	\bibitem{BarlowCerny}
	{\sc Barlow, M.~T.,  and Černý, J.}
	 Convergence to fractional kinetics for random walks associated with unbounded conductances.
	 {\em Probab. Theory Related Fields 4149}, 3 (2011), 639--673.
	
	\bibitem{BenArousCerny05}
	{\sc Ben Arous, G., and Černý, J.}
	 Bouchaud model exhibits two different aging regimes in dimension one.
	 {\em Ann. Probab. 15}, 2 (2005), 1161--1192.
	
	\bibitem{ben-arous_cerny_dynamics_2006}
	{\sc Ben Arous, G., and Černý, J.}
	 Dynamics of trap models.
	 In {\em Mathematical statistical physics \/} (Elsevier B. V., Amsterdam, 
	2006),  pp.~331--394

	\bibitem{bogachev_measure_2007}
	{\sc Bogachev, V.~I.}
	 {\em Measure theory. Vol. I, II.}
	 Springer--Verlag, Berlin, 2007. 
	
	\bibitem{Bou92}
	{\sc Bouchaud, J.~P.}
	 Weak ergodicity breaking and aging in disordered systems.
	 In {\em  J. Phys. I \/} (France),  2(1992) 1705--1713	
	
	\bibitem{cernyEJP}
	{\sc Černý, J.}
	 On two-dimensional random walk among heavy-tailed conductances.
	 {\em Electron. J. Probab. 16\/}, 10 (2011),  293--313.
	
	\bibitem{chen_time-2017}
	{\sc Chen, Z.~Q.}
	 Time fractional equations and probabilistic represntation.
	 {\em Chaos, Solitons \& Fractals 102} (2017), 168--174.

			\bibitem{Chiarini_Flo_sau}
	{\sc Chiarini, A., Floreani, S. and Sau, F. }
	\newblock From quenched 
	invariance principle
	to semigroup convergence with 
	applications to  exclusion processes.
	\newblock {\em  Electron. Commun. Probab. 29}, (2023), 1--17.
	
	
	\bibitem{croydon2017time}
	{\sc Croydon, D., Hambly, B., and  Kumagai, T.}
	 Time-changes of stochastic processes associated with resistance forms.
	 {\em Electron. J. Probab. 22\/} (2017), 1--41.
	
	\bibitem{croydon2019heat}
	{\sc Croydon, D.~A., Hambly, B.~M., and Kumagai, T.}
	 Heat kernel estimates for FIN processes associated with resistance forms.
	 {\em Stoch. Proc. Appl. 129}, 9 (2019), 2991--3017.
	
	\bibitem{etheridge2011}
	{\sc Etheridge, A.}
	 {\em Some Mathematical Models from Population Genetics: École D'Été de Probabilités de Saint-Flour XXXIX-2009.}
	 (Vol. 2012) Springer Science \& Business Media, 2011. 
	
	
	
	\bibitem{faggionato_bulk_2007}
	{\sc Faggionato, A.}
	 Bulk diffusion of 1{D} exclusion process with bond disorder.
	 {\em Markov Process. Related Fields 13}, 3 (2007), 519--542.
	 
	
	\bibitem{faggionato_cluster_2008}
	{\sc Faggionato, A.}
	 Random walks and exclusion processes among random conductances on
	random infinite clusters: homogenization and hydrodynamic limit.
	 {\em Electron. J. Probab. 13\/} (2008), no. 73, 2217--2247.
	
	\bibitem{faggionato2022hydrodynamic}
	{\sc Faggionato, A.}
	 Hydrodynamic limit of simple exclusion processes in symmetric random environments via duality and homogenization.
	 {\em Probab. Theory Related Fields 184}, 3 (2022), 1093--1137.
	 
	  \bibitem{faggionato_graphical_2023}
	 {\sc Faggionato, A.}
	 Graphical constructions of simple exclusion processes with applications to random
	 environments. {\em arXiv.2304.07703} (2023).
	 
	
	\bibitem{faggionato_hydrodynamic_2009}
	{\sc Faggionato, A., Jara, M., and Landim, C.}
	 Hydrodynamic behavior of 1{D} subdiffusive exclusion processes with
	random conductances.
	 {\em Probab. Theory Related Fields 144}, 3-4 (2009), 633--667.

	\bibitem{FloreaniJansenRedigWagner}
	{\sc Floreani, S., Jansen, S., Redig, F., and Wagner, S.}
	 Intertwining and duality for consistent Markov processes.
	 {\em Electron. J. Probab. 29}, (2024), 1--34.
	
	\bibitem{floreani_HDL_2021}
	{\sc Floreani, S., Redig, F., and Sau, F.}
	 Hydrodynamics for the partial exclusion process in random environment.
	 {\em Stoch. Proc. Appl. 142}, (2021), 124--158.
	
	
	
	\bibitem{FIN02}
	{\sc Fontes, L.~R.~G., Isopi, M., and Newman, M.}
	 Random walks with strongly inhomogeneous rates and singular diffusions: convergence, localization and aging in one dimension.
	 {\em Ann. Probab. 30}, 2 (2002), 579--604.
	
	
	\bibitem{FrancoLandim2010}
	{\sc Franco, T., and Landim, C.}
	 Hydrodynamic Limit of Gradient Exclusion Processes with Conductances.
	 {\em Arch. Rational Mech. Anal. 195},  (2010), 409--439.
	
	\bibitem{friedli_statistical_2017}
	{\sc Friedli, S., and Velenik, Y.}
	 {\em Statistical mechanics of lattice systems: a concrete mathematical introduction.}
	 Cambridge University Press, Cambridge, 2018. 
	
	\bibitem{giardina_duality_2009}
	{\sc Giardin\`a, C., Kurchan, J., Redig, F., and Vafayi, K.}
	 Duality and hidden symmetries in interacting particle systems.
	 {\em J. Stat. Phys. 135}, 1 (2009), 25--55.
	
	
	\bibitem{goncalves_scaling_2008}
	{\sc Gon\c{c}alves, P., and Jara, M.}
	 Scaling limits for gradient systems in random environment.
	 {\em J. Stat. Phys. 131}, 4 (2008), 691--716.
	
	\bibitem{jara_hydrodynamic_2011}
	{\sc Jara, M.}
	 Hydrodynamic {Limit} of the {Exclusion} {Process} in {Inhomogeneous}
	{Media}.
	 In {\em Dynamics, {Games} and {Science} {II}\/} (Berlin, Heidelberg,
	2011), M.~M. Peixoto, A.~A. Pinto, and D.~A. Rand, Eds., Springer Berlin
	Heidelberg, pp.~449--465.
	
	\bibitem{jara_quenched_2008}
	{\sc Jara, M., and Landim, C.}
	 Quenched non-equilibrium central limit theorem for a tagged particle
	in the exclusion process with bond disorder.
	 {\em Ann. Inst. Henri Poincar\'{e} Probab. Stat. 44}, 2 (2008),
	341--361.
	
	\bibitem{JaraLandimTeixeira}
	{\sc Jara, M., Landim, C., and Teixeira, A.}
	 Quenched scaling limits of trap models.
	 {\em Ann. Probab. 39}, 1 (2011), 176--223.
	
	\bibitem{kipnis_scaling_1999}
	{\sc Kipnis, C., and Landim, C.}
	 {\em Scaling limits of interacting particle systems}, vol.~320 of
	{\em Grundlehren der Mathematischen Wissenschaften}.
	 Springer-Verlag, Berlin, 1999.
	
	
	\bibitem{liggett_interacting_2005-1}
	{\sc Liggett, T.~M.}
	 {\em Interacting particle systems}.
	 Classics in Mathematics. Springer-Verlag, Berlin, 2005.
	 Reprint of the 1985 original.
	
	\bibitem{nagy_symmetric_2002}
	{\sc Nagy, K.}
	 Symmetric random walk in random environment in one dimension.
	 {\em Period. Math. Hungar. 45}, 1-2 (2002), 101--120.
	
	\bibitem{redig_symmetric_2020}
	{\sc Redig, F., Saada, E., and Sau, F.}
	 Symmetric simple exclusion process in dynamic environment:
	hydrodynamics.
	 {\em Electron. J. Probab. 25\/} (2020), Paper No. 138, 47.
	
	\bibitem{redig_sau_2018}
	{\sc Redig, F., and Sau, F.}
	 Factorized Duality, Stationary Product Measures and
	Generating Functions.
	 {\em  J. Stat. Phys. 172\/}, 4 (2018),  980--1008.
	
	
	\bibitem{schutz_non-abelian_1994}
	{\sc Sch\"{u}tz, G., and Sandow, S.}
	 Non-{Abelian} symmetries of stochastic processes: {Derivation} of
	correlation functions for random-vertex models and
	disordered-interacting-particle systems.
	 {\em Phys. Rev. E Stat. Phys. Plasmas Fluids Relat. Interdiscip.
		Topics 49}, 4 (1994), 2726--2741.
	
	\bibitem{Sokolov_etAL02}
	{\sc Sokolov, I.-M.,  Klafter, J. and Blumen, A.}
	 Fractional kinetics.
	 {\em  Physics Today 55\/}, 11(2002),  48--54.
	
	\bibitem{Valentim2011}
	{\sc Valentim, F.-J.}
	 Hydrodynamic limit of a $d$-dimensional exclusion process with conductances.
	 {\em Ann. Inst. Henri Poincar\'{e} Probab. Stat. 48}, 1 (2012),
	188–-211.	
	
\end{thebibliography}
\end{document}